\documentclass[11pt,twoside,a4paper]{article}
\usepackage{amscd,amsmath,amssymb,amsthm,latexsym,stmaryrd,authblk,mathabx,shuffle,hyperref,tikz,tkz-tab} 
\usepackage[latin1]{inputenc}
\usepackage[T1]{fontenc}   
\usepackage[english]{babel}
\input{xy}
\xyoption{all}

\theoremstyle{plain}
\newtheorem{theo}{Theorem}[section]
\newtheorem{lemma}[theo]{Lemma}
\newtheorem{cor}[theo]{Corollary}
\newtheorem{prop}[theo]{Proposition}
\newtheorem{defi}[theo]{Definition}

\theoremstyle{remark}
\newtheorem{remark}{Remark}[section]
\newtheorem{notation}{Notations}[section]
\newtheorem{example}{Example}[section]

\newcommand{\typ}[1]{\substack{\vspace{-2mm}\\ \begin{tikzpicture}[line cap=round,line join=round,>=triangle 45,x=0.7cm,y=0.7cm]
\clip(1.8,1.9) rectangle (2.2,2.7);
\draw [line width=.5pt,dash pattern=on #1pt off #1pt] (2.,2.5)-- (2.,2.);
\begin{scriptsize}
\end{scriptsize}
\end{tikzpicture}}}

\newcommand{\tun}{\begin{tikzpicture}[line cap=round,line join=round,>=triangle 45,x=0.5cm,y=0.5cm]
\clip(1.8,1.9) rectangle (2.2,2.2);
\begin{scriptsize}
\draw [fill=black] (2.,2.) circle (1pt);
\end{scriptsize}
\end{tikzpicture}}

\newcommand{\tdeux}[1]{\begin{tikzpicture}[line cap=round,line join=round,>=triangle 45,x=0.7cm,y=0.7cm]
\clip(1.8,1.9) rectangle (2.2,2.7);
\draw [line width=.5pt,dash pattern=on #1pt off #1pt] (2.,2.5)-- (2.,2.);
\begin{scriptsize}
\draw [fill=black] (2.,2.) circle (1pt);
\draw [fill=black] (2.,2.5) circle (1pt);
\end{scriptsize}
\end{tikzpicture}}

\newcommand{\ttroisun}[2]{\begin{tikzpicture}[line cap=round,line join=round,>=triangle 45,x=0.7cm,y=0.7cm]
\clip(1.5,1.9) rectangle (2.5,2.7);
\draw [line width=0.5pt,dash pattern=on #1pt off #1pt] (2.,2.)-- (1.7,2.5);
\draw [line width=0.5pt,dash pattern=on #2pt off #2pt] (2.,2.)-- (2.3,2.5);
\begin{scriptsize}
\draw [fill=black] (1.7,2.5) circle (1pt);
\draw [fill=black] (2.,2.) circle (1pt);
\draw [fill=black] (2.3,2.5) circle (1pt);
\end{scriptsize}
\end{tikzpicture}}

\newcommand{\tdun}[1]{\begin{tikzpicture}[line cap=round,line join=round,>=triangle 45,x=0.7cm,y=0.7cm]
\clip(1.8,1.9) rectangle (2.5,2.5);
\begin{scriptsize}
\draw [fill=black] (2.,2.) circle (1pt);
\end{scriptsize}
\draw(2.3,2.1) node {\tiny #1};
\end{tikzpicture}}

\newcommand{\tddeux}[3]{\begin{tikzpicture}[line cap=round,line join=round,>=triangle 45,x=0.7cm,y=0.7cm]
\clip(1.8,1.9) rectangle (2.5,3.);
\draw [line width=.5pt, dash pattern=on #3pt off #3pt] (2.,2.5)-- (2.,2.);
\begin{scriptsize}
\draw [fill=black] (2.,2.) circle (1pt);
\draw [fill=black] (2.,2.5) circle (1pt);
\end{scriptsize}
\draw(2.3,2.1) node {\tiny #1};
\draw(2.3,2.6) node {\tiny #2};
\end{tikzpicture}}

\newcommand{\tdtroisun}[5]{\begin{tikzpicture}[line cap=round,line join=round,>=triangle 45,x=0.7cm,y=0.7cm]
\clip(1.5,1.9) rectangle (2.5,3.2);
\draw [line width=0.5pt, dash pattern=on #4pt off #4pt] (2.,2.)-- (1.7,2.5);
\draw [line width=0.5pt, dash pattern=on #5pt off #5pt] (2.,2.)-- (2.3,2.5);
\begin{scriptsize}
\draw [fill=black] (1.7,2.5) circle (1pt);
\draw [fill=black] (2.,2.) circle (1pt);
\draw [fill=black] (2.3,2.5) circle (1pt);
\end{scriptsize}
\draw(2.4,2.1) node {\tiny #1};
\draw(2.3,2.8) node {\tiny #2};
\draw(1.7,2.8) node {\tiny #3};
\end{tikzpicture}}
\newcommand{\tdtroisdeux}[5]{\begin{tikzpicture}[line cap=round,line join=round,>=triangle 45,x=0.7cm,y=0.7cm]
\clip(1.8,1.9) rectangle (2.5,3.5);
\draw [line width=0.5pt, dash pattern=on #4pt off #4pt] (2.,2.5)-- (2.,2.);
\draw [line width=0.5pt, dash pattern=on #5pt off #5pt] (2.,2.5)-- (2.,3.);
\begin{scriptsize}
\draw [fill=black] (2.,2.) circle (1pt);
\draw [fill=black] (2.,2.5) circle (1pt);
\draw [fill=black] (2.,3.) circle (1pt);
\end{scriptsize}
\draw(2.3,2.1) node {\tiny #1};
\draw(2.3,2.6) node {\tiny #2};
\draw(2.3,3.1) node {\tiny #3};
\end{tikzpicture}}

\input{xy}
\xyoption{all}

\title{Algebraic structures on typed decorated rooted trees}
\date{}
\author{Lo\"\i c Foissy}
\affil{\small{Univ. Littoral Côte d'Opale, UR 2597
LMPA, Laboratoire de Mathématiques Pures et Appliquées Joseph Liouville
F-62100 Calais, France}.\\ Email: \texttt{foissy@univ-littoral.fr}}

\newcommand{\K}{\mathbb{K}}
\renewcommand{\L}{\mathcal{L}}

\newcommand{\T}{\mathcal{T}}

\newcommand{\D}{\mathcal{D}}

\newcommand{\g}{\mathfrak{g}}

\newcommand{\N}{\mathbb{N}}
\newcommand{\adm}{\mathrm{Adm}}

\newcommand{\h}{\mathcal{H}}

\newcommand{\bfT}{\mathbb{T}}
\newcommand{\bfF}{\mathbb{F}}
\renewcommand{\P}{\mathcal{P}}
\newcommand{\M}{\mathcal{M}}

\newcommand{\dec}{\mathrm{dec}}
\newcommand{\id}{\mathrm{Id}}
\newcommand{\type}{\mathrm{type}}
\renewcommand{\ker}{\mathrm{Ker}}
\newcommand{\im}{\mathrm{Im}}

\begin{document}

\maketitle

\tableofcontents

\begin{abstract}
Typed decorated trees are used by Bruned, Hairer and Zambotti to give a description of a renormalisation process
on stochastic PDEs. We here study the algebraic structures on these objects: multiple prelie algebras
and related operads (generalizing a result by Chapoton and Livernet), 
noncommutative and cocommutative Hopf algebras (generalizing Grossman and Larson's construction),
commutative and noncocommutative Hopf algebras (generalizing Connes and Kreimer's construction),
bialgebras in cointeraction (generalizing Calaque, Ebrahimi-Fard and Manchon's result). 
We also define families of morphisms and in particular we prove that any Connes-Kreimer Hopf algebra
of typed and decorated trees is isomorphic to a Connes-Kreimer Hopf algebra of non--typed and decorated
trees (the set of decorations of vertices being bigger), through a contraction process,
and finally obtain the Bruned-Hairer-Zambotti construction as a subquotient.
\end{abstract}



\section*{Introduction}

Bruned, Hairer and Zambotti used in \cite{Bruned,Hairer} typed trees in an essential way
to give a systematic description of a canonical renormalisation procedure of
stochastic PDEs. Typed trees are rooted trees  in which edges are decorated by elements of a fixed set $\T$ of types.
They also appear in a context of low dimension topology in \cite{Mathar} (there, described as nested parentheses)
and for the description
of combinatorial species in \cite{Bergeron}.
We here study several algebraic structures on these trees, generalizing results of Connes and Kreimer \cite{ConnesKreimer},
Chapoton and Livernet \cite{ChapotonLivernet}, Grossman and Larson \cite{GrossmanLarson}, 
Calaque, Ebrahimi-Fard and Manchon \cite{Manchon}.

In the work of Bruned, Hairer and Zambotti, the considered trees are typed, with a finite set of types
denoted by $\mathfrak{L}$, and labeled. We here forget about the labels and study the algebraic structures
induced by types.
We first define grafting products of trees, similar to the pre-Lie product of \cite{Chapoton}. 
For any type $t$, we obtain a pre-Lie product $\bullet_t$ on the space $\g_{\D,\T}$ of $\T$-typed trees
which vertices are decorated by elements of a set $\D$. 
For example, if $\typ{10}$ and $\typ{2}$ are two types, if $a$, $b$, $c\in \D$, then:
\begin{align*}
\tddeux{$a$}{$b$}{10}\bullet_{\typ{10}} \tdun{$c$}&=
\tdtroisun{$a$}{$c$}{$b$}{10}{10}+\tdtroisdeux{$a$}{$b$}{$c$}{10}{10},&
\tddeux{$a$}{$b$}{10}\bullet_{\typ{2}} \tdun{$c$}&=
\tdtroisun{$a$}{$c$}{$b$}{10}{2}+\tdtroisdeux{$a$}{$b$}{$c$}{10}{2}.
\end{align*}
Then $\g_{\D,\T}$, equipped with all these products, is a $\T$-multiple pre-Lie algebra
(Definition \ref{defi3}), also called matching pre-Lie algebras in \cite{ZhangGaoGuo}:
for any types $t$ and $t'$, for any $x,y,z\in \g_{\D,\T}$,
\[x\bullet_{t'}(y \bullet_t z)-(x\bullet_{t'} y)\bullet_t z=x\bullet_t(z \bullet_{t'} y)-(x\bullet_t z)\bullet_{t'} y.\]
We prove in Corollary \ref{cor9} that 
it is the free $\T$-multiple pre-Lie algebra generated by $\D$, generalizing the result of \cite{ChapotonLivernet}.
Consequently, we obtain a combinatorial description of the operad of $\T$-multiple pre-Lie algebras
in terms of $\T$-typed trees with indexed vertices (Theorem \ref{theo11}): for example,
\begin{align*}
\tddeux{$1$}{$2$}{10}\circ_1\tddeux{$1$}{$2$}{2}&=\tdtroisun{$1$}{$3$}{$2$}{2}{10}
+\tdtroisdeux{$1$}{$2$}{$3$}{2}{10},&
\tddeux{$1$}{$2$}{2}\circ_2\tddeux{$1$}{$2$}{10}&=\tdtroisdeux{$1$}{$2$}{$3$}{2}{10}.
\end{align*}
We also give a desription of the Koszul dual operad and of its free algebras in Propositions
\ref{prop12} and \ref{prop13}, generalizing a result of \cite{Chapoton}. 

For any family $\lambda=(\lambda_t)_{t\in \T}$ with a finite support, the product 
$\bullet_\lambda=\sum \lambda_t\bullet_t$ is pre-Lie: 
using the Guin-Oudom construction \cite{Guin2,Guin1}, we obtain a Hopf algebraic structure
$\h_{\D,\T}^{GL_\lambda}=(S(\g_{\D,\T}),\star_\lambda,\Delta)$ 
on the symmetric algebra generated by $\T$-typed and $\D$-decorated trees,
that is to say on the space of $\T$-typed and $\D$-decorated forests. The coproduct $\Delta$
is given by partitions of forests into two forests and the $\star_\lambda$ product is given by grafting.
For example:
\begin{align*}
\tddeux{$a$}{$b$}{10}\star_\lambda \tdun{$c$}&=\tddeux{$a$}{$b$}{10}\tdun{$c$}
+\lambda_{\typ{10}}\tdtroisun{$a$}{$c$}{$b$}{10}{10}
+\lambda_{\typ{10}}\tdtroisdeux{$a$}{$b$}{$c$}{10}{10}
+\lambda_{\typ{2}}\tdtroisun{$a$}{$c$}{$b$}{10}{2}
+\lambda_{\typ{2}}\tdtroisdeux{$a$}{$b$}{$c$}{10}{2}.
\end{align*}
In the non-typed case, we get back the Grossman-Larson Hopf algebra of trees \cite{GrossmanLarson}.
Dually, we obtain Hopf algebras $\h_{\D,\T}^{CK_\lambda}$, generalizing the Connes-Kreimer Hopf algebra
\cite{ConnesKreimer} of rooted trees. For example:
\begin{align*}
\Delta^{CK_\lambda}(\tddeux{$a$}{$b$}{10})&=\tddeux{$a$}{$b$}{10}\otimes 1+1\otimes \tddeux{$a$}{$b$}{10}
+\lambda_{\typ{10}} \tdun{$a$}\otimes \tdun{$b$},\\
\Delta^{CK_\lambda}(\tdtroisun{$a$}{$c$}{$b$}{10}{10})
&=\tdtroisun{$a$}{$c$}{$b$}{10}{10}\otimes 1+1\otimes \tdtroisun{$a$}{$c$}{$b$}{10}{10}
+\lambda_{\typ{10}} \tddeux{$a$}{$b$}{10}\otimes \tdun{$c$}
+\lambda_{\typ{10}} \tddeux{$a$}{$c$}{10}\otimes \tdun{$b$}
+\lambda_{\typ{10}}^2 \tdun{$a$}\otimes \tdun{$b$}\tdun{$c$},\\
\Delta^{CK_\lambda}(\tdtroisun{$a$}{$c$}{$b$}{10}{2})
&=\tdtroisun{$a$}{$c$}{$b$}{10}{2}\otimes 1+1\otimes \tdtroisun{$a$}{$c$}{$b$}{10}{2}
+\lambda_{\typ{2}} \tddeux{$a$}{$b$}{10}\otimes \tdun{$c$}
+\lambda_{\typ{10}} \tddeux{$a$}{$c$}{2}\otimes \tdun{$b$}
+\lambda_{\typ{10}}\lambda_{\typ{2}} \tdun{$a$}\otimes \tdun{$b$}\tdun{$c$}.
\end{align*}
This Hopf algebra satisfies a universal property in Hochschild cohomology, as does the Connes-Kreimer's Hopf algebra.
We describe it in the simpler case where $\T$ is finite (Theorem \ref{theo22}).
We finally give a second coproduct $\delta$ on $\h_{\D,\T}^{CK_\lambda}$, such that 
$\h_{\D,\T}^{CK_\lambda}$ is a Hopf algebra in the category
of $(S(\g_{\D,\T}),m,\delta)$-right comodules, generalizing the result of \cite{Manchon}.
This coproduct $\delta$ is given by a contraction-extraction process. For example, in the non-decorated case:
\begin{align*}
\delta(\tun)&=\tun \otimes \tun,\\
\delta(\tdeux{10})&=\tdeux{10} \otimes \tun\tun+\tun \otimes \tdeux{10},\\
\delta(\ttroisun{10}{10})&=\ttroisun{10}{10}\otimes \tun\tun\tun
+2\tdeux{10}\otimes \tdeux{10}\tun+\tun \otimes \ttroisun{10}{10},\\
\delta(\ttroisun{10}{2})&=\ttroisun{10}{2}\otimes \tun\tun\tun
+\tdeux{10}\otimes \tdeux{2}+\tdeux{2}\otimes \tdeux{10}\tun+\tun \otimes \ttroisun{10}{2}.
\end{align*}

We are also interested in morphisms between these objects. 
We prove that if $\lambda$ and $\mu$ are both nonzero, then the pre-Lie algebras $(\g_{\D,\T},\bullet_\lambda)$
and $(\g_{\D,\T},\bullet_\mu)$ are isomorphic (Corollary \ref{cor18}). Consequently,
if $\lambda$ and $\mu$ are both nonzero, the Hopf algebras $\h_{\D,\T}^{GL_\lambda}$
and $\h_{\D,\T}^{GL_\mu}$ are isomorphic; dually, the Hopf algebras $\h_{\D,\T}^{CK_\lambda}$
and $\h_{\D,\T}^{CK_\mu}$ are isomorphic (Corollary \ref{cor24}). 
Using Livernet's rigidity theorem \cite{Livernet} and a nonassociative permutative coproduct
defined in Proposition \ref{prop14}, we prove that if $\lambda\neq 0$,
then $(\g_{\D,\T},\bullet_\lambda)$ is, as a pre-Lie algebra, freely generated by a family of typed trees
$\D'=\T_{\D,\T}^{(t_0)}$ satisfying a condition on the type of edges born from the root (Corollary \ref{cor15}).
As a consequence, the Hopf algebra $\h_{\D,\T}^{CK_\lambda}$ of typed and decorated trees is isomorphic
to a Connes-Kreimer Hopf algebra of non typed and decorated trees $\h_{\D'}^{CK}$,
and an explicit isomorphism is described with the help of contraction in Proposition \ref{prop26}.\\

This paper is organized as follows: the first section gives the basic definition of typed rooted trees
and enumeration results, when the number of types and decorations are finite.
The second section is about the $\T$-multiple pre-Lie algebra structures on these trees and the underlying operads..
The freeness of the pre-Lie structures on typed decorated trees and its consequences are studied in the third section.
In the last section, the dual Hopf algebras $\h_{\D,\T}^{GL_\lambda}$ and $\h_{\D,\T}^{CK_\lambda}$ are defined
and studied and related to the constructions of Bruned, Hairer and Zambotti \cite{Bruned,Hairer}: 
forgetting the labels, the two coproducts they use on a family of typed and partially decorated trees
are a subquotient of a the construction presented here. \\

\begin{notation} \begin{itemize}
\item We denote by $\K$ a commutative field of characteristic zero. All the objects (vector spaces, algebras,
 coalgebras, pre-Lie algebras$\ldots$)
in this text will be taken over $\K$.
\item For any $n\in \mathbb{N}$, we denote by $[n]$ the set $\{1,\ldots,n\}$.
\item For any set $\T$, we denote by $\K^\T$ the set of family $\lambda=(\lambda_t)_{t\in \T}$ of elements of $\K$
indexed by $\T$, and we denote by $\K^{(\T)}$ the set of elements $\lambda\in \K^\T$ with a finite support.
Note that if $\T$ is finite, then $\K^\T=\K^{(\T)}$. 
\end{itemize}\end{notation}

\section{Typed decorated trees}

\subsection{Definition}

\begin{defi}
Let $\D$ and $\T$ be two nonempty sets. 
\begin{enumerate}
\item A $\D$-decorated $\T$-typed forest is a triple $(F,\dec,\type)$, where:
\begin{itemize}
\item $F$ is a rooted forest. The set of its vertices is denoted by $V(F)$ and the set of its edges by $E(F)$.
\item $\dec:V(F)\longrightarrow \D$ is a map.
\item $\type:E(F)\longrightarrow \T$ is a map.
\end{itemize} 
If the underlying rooted forest of $F$ is connected, we shall say that $F$ is a $\D$-decorated $\T$-typed tree. 
\item If $(F,\dec_F,\type_F)$ and $(G,\dec_G,\type_G)$ are two $\D$-decorated  $\T$-typed  forests,
they are isomorphic if there exists a rooted forest isomorphism $f$ from $F$ to $G$ such that for any vertex $v$ of $F$,
$\dec_G(f(v))=\dec_F(v)$ and for any edge $e$ of $F$, $\type_G(f(e))=\type_F(e)$
\item For any finite set $A$, we denote by $\bfT_\T(A)$ the set of $A$-decorated 
$\T$-typed trees $T$ such that $V(T)=A$
and $\dec=\id_A$, and by $\bfF_\T(A)$ the set of $A$-decorated $\T$-typed forests $F$ such that $V(F)=A$ and $\dec=\id_A$.
\item For any $n\geq 0$, we denote by $\bfT_{\D,\T}(n)$ the set of isomorphism classes of $\D$-decorated $\T$-typed trees $T$ 
such that $|V(T)|=n$ and by $\bfF_{\D,\T}(n)$ the set of isomorphism classes of $\D$-decorated $\T$-typed forests $F$ 
such that $|V(F)|=n$. 
We also put:
\begin{align*}
\bfT_{\D,\T}&=\bigsqcup_{n\geq 0} \bfT_{\D,\T}(n),&\bfF_{\D,\T}&=\bigsqcup_{n\geq 0} \bfF_{\D,\T}(n).
\end{align*}
\end{enumerate}
\end{defi}

\begin{example}
We shall represent the decorations of the vertices by letters alongside them.  
If $\T$ contains two elements, represented by $\typ{10}$ and $\typ{2}$, then:
\begin{align*}
\bfF_{\D,\T}(1)&=\{\tdun{$d$},d\in \D\},\\
\bfF_{\D,\T}(2)&=\left\{\begin{array}{c}
\tdun{$a$}\tdun{$b$},\tddeux{$a$}{$b$}{10},\tddeux{$a$}{$b$}{2}, a,b\in \D
\end{array}\right\},\\
\bfF_{\D,\T}(3)&=\left\{\begin{array}{c}
\tdun{$a$}\tdun{$b$}\tdun{$c$},\tddeux{$a$}{$b$}{10}\tdun{$c$},\tddeux{$a$}{$b$}{2}\tdun{$c$},
\tdtroisun{$a$}{$c$}{$b$}{10}{10},\tdtroisun{$a$}{$c$}{$b$}{2}{10},\tdtroisun{$a$}{$c$}{$b$}{2}{2},
\tdtroisdeux{$a$}{$b$}{$c$}{10}{10},\tdtroisdeux{$a$}{$b$}{$c$}{2}{10},
\tdtroisdeux{$a$}{$b$}{$c$}{10}{2},\tdtroisdeux{$a$}{$b$}{$c$}{2}{2}, a,b,c\in \D\end{array}\right\}.
\end{align*}
Note that for any $a$, $b$, $c\in \D$:
\begin{align*}
\tdun{$a$}\tdun{$b$}&=\tdun{$b$}\tdun{$a$},&
\tdtroisun{$a$}{$c$}{$b$}{10}{10}&=\tdtroisun{$a$}{$b$}{$c$}{10}{10},&
\tdtroisun{$a$}{$c$}{$b$}{10}{2}&=\tdtroisun{$a$}{$b$}{$c$}{2}{10},&
\tdtroisun{$a$}{$c$}{$b$}{2}{2}&=\tdtroisun{$a$}{$b$}{$c$}{2}{2}.
\end{align*}
Moreover:
\begin{align*}
\bfF_{\D,\T}([1])&=\{\tdun{$1$}\},\\
\bfF_{\D,\T}([2])&=\left\{\begin{array}{c}\tdun{$1$}\tdun{$2$},
\tddeux{$1$}{$2$}{10},\tddeux{$2$}{$1$}{10},\tddeux{$1$}{$2$}{2},\tddeux{$2$}{$1$}{2}\end{array}\right\},\\
\bfF_{\D,\T}([3])&=\left\{\begin{array}{l}
\tdun{$1$}\tdun{$2$}\tdun{$3$},
\tddeux{$1$}{$2$}{10}\tdun{$3$},\tddeux{$1$}{$3$}{10}\tdun{$2$},\tddeux{$2$}{$1$}{10}\tdun{$3$},
\tddeux{$2$}{$3$}{10}\tdun{$1$},\tddeux{$3$}{$1$}{10}\tdun{$2$},\tddeux{$3$}{$2$}{10}\tdun{$1$},\\
\tddeux{$1$}{$2$}{2}\tdun{$3$},\tddeux{$1$}{$3$}{2}\tdun{$2$},\tddeux{$2$}{$1$}{2}\tdun{$3$},
\tddeux{$2$}{$3$}{2}\tdun{$1$},\tddeux{$3$}{$1$}{2}\tdun{$2$},\tddeux{$3$}{$2$}{2}\tdun{$1$},\\
\tdtroisun{$1$}{$3$}{$2$}{10}{10},\tdtroisun{$2$}{$3$}{$1$}{10}{10},\tdtroisun{$3$}{$2$}{$1$}{10}{10},
\tdtroisun{$1$}{$3$}{$2$}{2}{10},\tdtroisun{$1$}{$2$}{$3$}{2}{10},\tdtroisun{$2$}{$3$}{$1$}{2}{10},
\tdtroisun{$2$}{$1$}{$3$}{2}{10},\tdtroisun{$3$}{$2$}{$1$}{2}{10},\tdtroisun{$3$}{$1$}{$2$}{2}{10},
\tdtroisun{$1$}{$3$}{$2$}{2}{2},\tdtroisun{$2$}{$3$}{$1$}{2}{2},
\tdtroisun{$3$}{$2$}{$1$}{2}{2},\\
\tdtroisdeux{$1$}{$2$}{$3$}{10}{10},\tdtroisdeux{$1$}{$3$}{$2$}{10}{10},\tdtroisdeux{$2$}{$1$}{$3$}{10}{10},
\tdtroisdeux{$2$}{$3$}{$1$}{10}{10},\tdtroisdeux{$3$}{$1$}{$2$}{10}{10},\tdtroisdeux{$3$}{$2$}{$1$}{10}{10},
\tdtroisdeux{$1$}{$2$}{$3$}{2}{10},\tdtroisdeux{$1$}{$3$}{$2$}{2}{10},
\tdtroisdeux{$2$}{$1$}{$3$}{2}{10},\tdtroisdeux{$2$}{$3$}{$1$}{2}{10},
\tdtroisdeux{$3$}{$1$}{$2$}{2}{10},\tdtroisdeux{$3$}{$2$}{$1$}{2}{10},\\
\tdtroisdeux{$1$}{$2$}{$3$}{10}{2},\tdtroisdeux{$1$}{$3$}{$2$}{10}{2},
\tdtroisdeux{$2$}{$1$}{$3$}{10}{2},\tdtroisdeux{$2$}{$3$}{$1$}{10}{2},
\tdtroisdeux{$3$}{$1$}{$2$}{10}{2},\tdtroisdeux{$3$}{$2$}{$1$}{10}{2},
\tdtroisdeux{$1$}{$2$}{$3$}{2}{2},\tdtroisdeux{$1$}{$3$}{$2$}{2}{2},
\tdtroisdeux{$2$}{$1$}{$3$}{2}{2},\tdtroisdeux{$2$}{$3$}{$1$}{2}{2},
\tdtroisdeux{$3$}{$1$}{$2$}{2}{2},\tdtroisdeux{$3$}{$2$}{$1$}{2}{2}
\end{array}\right\}.
\end{align*}
\end{example}

\begin{remark}
If $|\T|=1$, all the edges of elements of $\bfF_{\D,\T}$ have the same type: 
we work with $\D$-decorated rooted forests.
In this case, we shall omit $\T$ in the indices describing the forests, trees, spaces we are considering.
\end{remark}

\subsection{Enumeration}

We assume here that $\D$ and $\T$ are finite, of respective cardinality $D$ and $T$. For all $n\geq 0$, we put:
\begin{align*}
t_{D,T}(n)&=|\bfT_{\T,\D}(n)|,& f_{D,T}(n)&=|\bfF_{\T,\D}(n)|,\\
T_{D,T}(X)&=\sum_{n=1}^\infty t_{D,T}(n)X^n,& F_{D,T}(X)&=\sum_{n=0}^\infty f_{D,T}(n)X^n.
\end{align*}
As any element of $\bfF_{\T,\D}$ can be uniquely decomposed as the disjoint union of its connected components, which are
elements of $\bfT_{\D,\T}$, we obtain:
\begin{align}
\label{EQ1} F_{D,T}(X)&=\prod_{n=1}^\infty \frac{1}{(1-X^n)^{t_{D,T}(n)}}.
\end{align}
We put $\T=\{t_1,\ldots,t_T\}$. For any $d\in \D$, we consider:
\begin{align*}
B_d&:\left\{\begin{array}{rcl}
(\bfF_{\D,\T})^T&\longrightarrow&\bfT_{\D,\T}\\
(F_1,\ldots,F_T)&\longrightarrow&B_d(F_1,\ldots,F_T),
\end{array}\right.
\end{align*}
where $B_d(F_1,\ldots,F_T)$ is the tree obtained by grafting the forests $F_1,\ldots,F_n$ 
on a common root decorated by $d$;
the edges from this root to the roots of $F_i$ are of type $t_i$ for any $1\leq i\leq T$.
Then $B_d$ is injective, homogeneous of degree $1$, and moreover $\bfT_{\D,\T}$ 
is the disjoint union of the $B_d((\bfF_{\D,\T})^T)$, $d\in \D$. Hence:
\begin{align}
\label{EQ2} T_{D,T}(X)&=DX (F_{D,T})^T=DX\prod_{n=1}^\infty\frac{1}{(1-X^n)^{t_{D,T}(n)T}}.
\end{align}
Note that (\ref{EQ2}) allows to compute $t_{D,T}(n)$ by induction on $n$, 
and (\ref{EQ1}) allows to deduce $f_{D,T}(n)$.

\begin{lemma}
For any $n\in \N$, 
\[t_{D,T}(n)=\frac{t_{TD,1}(n)}{T}.\]
\end{lemma}

\begin{proof} By induction on $n$. If $n=1$, $t_{D,T}(1)=D$ and $t_{TD,1}=TD$, which gives the result. 
Let us assume the result at all ranks $k<n$.
Then $t_{D,T}(n)T$ is the coefficient of $X^n$ in:
\[TDX\prod_{k=1}^{n-1}\frac{1}{(1-X^k)^{t_{D,T}(k)T}}=TDX\prod_{k=1}^{n-1}\frac{1}{(1-X^k)^{t_{TD,1}(k)}},\]
which is precisely $t_{TD,1}(n)$. \end{proof}

\begin{example}
We obtain:
\begin{align*}
t_{D,T}(1)&=D,\\
t_{D,T}(2)&=D^2T,\\
t_{D,T}(3)&=\frac{D^2T(3DT+1)}{2},\\
t_{D,T}(4)&=\frac{D^2T(8D^2T^2+3DT+1)}{3},\\
t_{D,T}(5)&=\frac{D^2T(125D^3T^3+54D^2T^2+31DT+6)}{24},\\
t_{D,T}(6)&=\frac{D^2T(162D^4T^4+80D^3T^3+45D^2T^2+10DT+3)}{15},\\
t_{D,T}(7)&=\frac{D^2T(16807D^5T^5+9375D^4T^4+5395D^3T^3+2025D^2T^2+838DT+120)}{720}
\end{align*}
Specializing, we find the following sequences of the OEIS \cite{Sloane}:
\[\begin{array}{c||c|c|c|c}
T\setminus D&1&2&3&4\\
\hline\hline  1&A0081&A038055&A038059&A136793\\
2&A00151&A136794\\
3&A006964\\
4&A052763\\
5&A052788\\
6&A246235\\
7&A246236\\
8&A246237\\
9&A246238\\
10&A246239
\end{array}\]\end{example}

We shall give tables of values of $t_{D,T}(k)$ in the appendix.

\section{Multiple pre-Lie algebras}

We here fix a nonempty set $\T$ of types of edges.

\subsection{Definition}

\begin{defi} \label{defi3}
A $\T$-multiple pre-Lie algebra is a family $(V,(\bullet_t)_{t\in \T})$, 
where $V$ is a vector space and for all $t\in \T$,
$\bullet_t$ is a bilinear product on $V$ such that:
\begin{align*}
&\forall t,t'\in \T,\:\forall x,y,z\in V,& 
x\bullet_{t'}(y \bullet_t z)-(x\bullet_{t'} y)\bullet_t z&=x\bullet_t(z \bullet_{t'} y)-(x\bullet_t z)\bullet_{t'} y.
\end{align*}
\end{defi}

\begin{remark}
For any $t\in \T$, $(V,\bullet_t)$ is a pre-Lie algebra. 
More generally, for any family $\lambda=(\lambda_t)_{t\in \T}\in \K^{(\T)}$,
putting $\bullet_\lambda=\sum \lambda_t \bullet_t$, $(V,\bullet_\lambda)$ is a pre-Lie algebra.
\end{remark}

\begin{prop}
Let $\D$ be any set; we denote by $\g_{\D,\T}$ the vector space generated by $\bfT_{\D,\T}$. 
For any $T,T'\in \bfT_{\D,\T}$,
$v\in V(T)$ and $t\in \T$, we denote by $T\bullet_t^{(v)} T'$ the $\D$-decorated $\T$-typed tree 
obtained by grafting $T'$ on $v$, the created edge being of type $t$. 
We then define a product $\bullet_t$ on $\g_{\D,\T}$ by:
\begin{align*}
&\forall T,T'\in \bfT_{\D,\T},&T\bullet_t T'&=\sum_{v\in V(T)} T\bullet_t^{(v)} T'.
\end{align*}
Then $(\g_{\D,\T},(\bullet_t)_{t\in \T})$ is a $\T$-multiple pre-Lie algebra.
\end{prop}

\begin{proof} Let $T,T',T''$ be elements of $\bfT_{\D,\T}$ and $t',t''\in \T$.
\begin{align*}
&(T\bullet_{t'} T')\bullet_{t''} T''-T\bullet_{t'}(T'\bullet_{t''}T'')\\
&=\sum_{v\in Vert(T), v'\in Vert(T)\sqcup Vert(T')} (T\bullet_{t'}^{(v)} T')\bullet_{t''}^{(v')} T''
-\sum_{v\in Vert(T),v'\in Vert(T')}T\bullet_{t'}^{(v)}(T'\bullet_{t''}^{(v')} T'')\\
&=\sum_{v\in Vert(T), v'\in Vert(T)} (T\bullet_{t'}^{(v)} T')\bullet_{t''}^{(v')} T''\\
&+\sum_{v\in Vert(T),v'\in Vert(T')}(T\bullet_{t'}^{(v)} T')\bullet_{t''}^{(v')} T''
-T\bullet_{t'}^{(v)}(T'\bullet_{t''}^{(v')} T'')\\
&=\sum_{v\in Vert(T), v'\in Vert(T)} (T\bullet_{t'}^{(v)} T')\bullet_{t''}^{(v')} T''\\
&=\sum_{v\in Vert(T), v'\in Vert(T)} (T\bullet_{t''}^{(v')} T'')\bullet_{t'}^{(v)} T'\\
&=(T\bullet_{t''} T'')\bullet_{t'} T'-T\bullet_{t''}(T''\bullet_{t'}T').
\end{align*}
So $\g_{\D,\T}$ is indeed a $\T$-multiple pre-Lie algebra.
\end{proof}

\begin{example} If $a$,$b$, $c\in \D$ and $\typ{10}$, $\typ{2} \in \T$:
\begin{align*}
\tddeux{$a$}{$b$}{10}\bullet_{\typ{10}} \tdun{$c$}&=
\tdtroisun{$a$}{$c$}{$b$}{10}{10}+\tdtroisdeux{$a$}{$b$}{$c$}{10}{10},&
\tddeux{$a$}{$b$}{10}\bullet_{\typ{2}} \tdun{$c$}&=
\tdtroisun{$a$}{$c$}{$b$}{10}{2}+\tdtroisdeux{$a$}{$b$}{$c$}{10}{2}.
\end{align*}\end{example}

\subsection{Guin-Oudom extension of the pre-Lie products}

\begin{notation}
Let $V$ be a vector space. We denote 
\[V^{\oplus \T}=\bigoplus_{t\in \T} V\delta_t.\]
\end{notation}

\begin{lemma}
If for any $t\in \T$, $\bullet_t$ is a bilinear product on a vector space $V$, 
we define $\bullet:(V^{\oplus \T})^{\otimes 2}\longrightarrow V^{\oplus \T}$ by:
\[x\delta t \bullet x'\delta_{t'}=(x\bullet_{t'} y)\delta_t.\]
Then $(V,(\bullet_t)_{t\in \T})$ is a $\T$-multiple pre-Lie algebra if, 
and only if, $(V^{\oplus \T},\bullet)$ is a pre-Lie algebra.
\end{lemma}

\begin{proof} Let $x,x',x''\in V$, $t,t',t''\in \T$. Then:
\begin{align*}
&x\delta_t\bullet (x'\delta_{t'}\bullet x''\delta_{t''})-(x\delta_t\bullet x'\delta_{t'})\bullet x''\delta_{t''}
=\left((x\bullet_{t'} x')\bullet_{t''} x''-x\bullet_{t'}(x'\bullet_{t''}x'')\right)\delta_t,
\end{align*}
which implies the result. \end{proof}

\begin{notation}
The symmetric algebra $S(V)$ is given its usual coproduct $\Delta$, making it a bialgebra:
\begin{align*}
&\forall x\in V,&\Delta(x)&=x\otimes 1+1\otimes x.
\end{align*}
We shall use Sweedler's notation: for any $w\in S(V)$, $\Delta(w)=\sum w^{(1)}\otimes w^{(2)}$.
\end{notation}

\begin{theo}
Let $V$ be a $\T$-multiple pre-Lie algebra. One can define a product 
\[\bullet:S(V)\otimes S\left(V^{\oplus\T}\right)\longrightarrow S(V)\]
in the following way: for any $u,v \in S(V)$, $w\in S\left(V^{\oplus\T}\right)$, $x\in V$, $t\in \T$,
\begin{align*}
1\bullet w&=\varepsilon(w),\\
u\bullet 1&=u,\\
uv\bullet w&=\sum\left (u\bullet w^{(1)}\right)\left(v\bullet w^{(2)}\right),\\
u\bullet w (x\delta_t)&=(u\bullet w)\bullet_t x-x\bullet (w\bullet_t x),
\end{align*}
where $\bullet_t$ is extended to $S(V)\otimes V$ and $S\left(V^{\oplus\T}\right)\otimes V$ by:
\begin{align*}
&\forall x_1,\ldots,x_k,x\in V,\:t_1,\ldots,t_k \in \T,&x_1\ldots x_k \bullet_t x
&=\sum_{i=1}^k x_1\ldots (x_i\bullet_t x)\ldots x_k,\\
&&(x_1\delta_{t_1})\ldots (x_k\delta_{t_k}) \bullet_t x
&=\sum_{i=1}^k (x_1\delta_{t_1})\ldots ((x_i\bullet_t x)\delta_{t_i})\ldots (x_k\delta_{t_k}).
\end{align*}\end{theo}

\begin{proof} \textit{Uniqueness}. The last formula allows to compute $x\bullet w$ for any $x\in V$ 
and $w\in S\left(V^{\oplus\T}\right)$ by induction on the length of $w$;
the other ones allow to compute $u\bullet w$ for any $u\in S(V)$ by induction on the length on $u$. 
So this product $\bullet$ is unique.\\

\textit{Existence}. Let us use the Guin-Oudom construction \cite{Guin1,Guin2} on the pre-Lie algebra $V^{\otimes \T}$. 
We obtain a product $\bullet$
defined on $S\left(V^{\oplus\T}\right)$ such that for any $u,v,w\in S\left(V^{\oplus\T}\right)$, $x\in V^{\oplus \T}$:
\begin{align*}
1\bullet w&=\varepsilon(w),\\
u\bullet 1&=u,\\
uv\bullet w&=\sum \left(u\bullet w^{(1)}\right)\left(v\bullet w^{(2)}\right),\\
u\bullet w x&=(u\bullet w)\bullet x-x\bullet (w\bullet x).
\end{align*}
Let $f:\T\longrightarrow \K$ be any nonzero map. We consider the surjective algebra morphism 
$F:S\left(V^{\oplus\T}\right)\longrightarrow S(V)$,
sending $x\delta_t$ to $f(t)x$ for any $x\in V$, $t\in \T$. Its kernel is generated by the elements 
$X_{t,t'}x=(f(t')\delta_t-f(t)\delta_{t'})x$,
where $x\in V$ and $t,t'\in \T$. We denote by $J$ the vector space generated by the elements $X_{t,t'}x$.
Let us prove that for any $w\in S\left(V^{\oplus\T}\right)$, $J\bullet w\subseteq J$ by induction on the length $n$ of $w$.
If $n=0$, we can assume that $w=1$ and this is obvious. If $n=1$, we can assume that $w=x'\delta_{t''}$. Then:
\begin{align*}
X_{t,t'}x\bullet w&=(f(t')\delta_t-f(t)\delta_{t'}) x\bullet_{t''} x'=X_{t,'t'}x\bullet_{t''} x' \in J.'
\end{align*}
Let us assume the result at rank $n-1$. We can assume that $w=w' x'\delta_t$, the length of $w'$ being $n-1$. 
For any $x\in J$:
\[x\bullet w=(x\bullet w')\bullet x'-x\bullet(w'\bullet x').\]
The length of $w'$ and $w'\bullet x'$ is $n-1$, so $x\bullet w'$ and $x\bullet (w'\bullet x')$ belong to $J$. 
From the case $n=1$,
$(x\bullet w')\bullet x'\in J$, so $x\bullet w\in J$.\\

For any $x\in J$, $u,v\in S\left(V^{\oplus\T}\right)$:
\[xu\bullet v=\underbrace{\left(x\bullet v^{(1)}\right)}_{\in J} \left(u\bullet v^{(2)}\right)  \in \ker(F).\]
This proves that $\ker(F)\bullet S\left(V^{\oplus\T}\right)\subseteq \ker(F)$. Hence, $\bullet$ induces a product also denoted by $\bullet$,
defined from $S(V)\otimes S(V^{\otimes \T})$ to $S(V)$. It is not difficult to show that it does not depend on the choice of $f$
and satisfies the required properties. \end{proof}

\begin{defi}
Let $d\in \D$, $T_1,\ldots,T_k\in \bfT_{\D,\T}$, $t_1,\ldots,t_k\in \T$. We denote by
\[B_d\left(\prod_{i\in [k]} T_i\delta_{t_i}\right)\]
the $\T$-typed $\D$-decorated tree obtained by grafting $T_1,\ldots,T_k$ on a common root decorated by $d$,
the edge between this root and the root of $T_i$ being of type $t_i$ for any $i$.
This defines a map $B_d:S\left(Vect(\bfT_{\D,\T})^{\oplus \T}\right)\longrightarrow S(Vect(\bfT_{\D,\T}))$.
\end{defi}

\begin{lemma} \label{lemma8}
For any $d\in \D$, $T_1,\ldots,T_k\in \bfT_{\D,\T}$, $t_1,\ldots,t_k\in \T$:
\[B_d\left(\prod_{i\in [k]} T_i\delta_{t_i}\right)=\tdun{$d$}\bullet\prod_{i\in [k]} T_i\delta_{t_i}.\]
\end{lemma}

\begin{proof} We write $\displaystyle F=\prod_{i\in [k]} T_i\delta_{t_i}$. We proceed by induction on $k$. 
If $k=0$, then $F=1$ and $\tdun{$d$}\bullet 1=\tdun{$d$}=B_d(1)$.
let us assume the result at rank $k-1$, with $k\geq 1$. We can write $F=F'T\delta_t$, with $length(F')=k-1$, $T=T_k$  and $t=t_k$.
Then:
\begin{align*}
\tdun{$d$}\bullet F&=(\tdun{$d$}\bullet F')\bullet T\delta_t-\tdun{$d$}\bullet (F'\bullet T\delta_t)\\
&=B_d(F')\bullet_t T-B_d(F'\bullet_t T)\\
&=B_d(F'T\delta_t)+B_d(F'\bullet_t T)-B_d(F'\bullet_t T)\\
&=B_d(F).
\end{align*}
So the result holds for all $k\geq 0$. \end{proof}

\begin{cor} \label{cor9}
Let $A$ be a $\T$-multiple pre-Lie algebra and, for any $d\in \D$, $a_d \in A$. 
There exists a unique $\T$-multiple algebra morphism
$\phi:\g_{\D,\T}\longrightarrow A$, such that for any $d\in \D$, $\phi(\tdun{$d$})=a_d$.
In other words, $\g_{\T,\D}$ is the free $\T$-multiple pre-Lie algebra generated by $\D$.
\end{cor}

\begin{proof} \textit{Uniqueness}. Using the Guin-Oudom product and lemma \ref{lemma8}, 
$\phi$ is the unique linear map inductively defined by:
\[\phi\left(B_d\left(\prod_{i\in[k]}T_i \delta_{t_i}\right)\right)
=a_d \bullet \prod_{i\in [k]} \phi(T_i) \delta_{t_i}.\]

\textit{Existence}. Let $T,T'\in \bfT_{\D,\T}$ and $t\in \T$. Let us prove that $\phi(T\bullet_t T')=\phi(T)\bullet_t \phi(T')$
by induction on $n=|T|$. If $n=1$, we assume that $T=\tdun{$d$}$. Then $T\bullet_t T'=B_d(T'\delta_t)$, so:
\begin{align*}
\phi(T\bullet_t T')&=a_d\bullet (\phi(T')\delta_t=a_d \bullet_t \phi(T')=\phi(T)\bullet_t \phi(T').
\end{align*}
Let us assume the result at all ranks $<|T|$. We put:
\[T=B_d\left(\prod_{i=1}^k T_i\delta_{t_i}\right).\]
By definition of the pre-Lie product of $\g_{\D,\T}$ in terms of grafting:
\begin{align*}
T\bullet T'&=B_d\left(\prod_{i=1}^k T_i \delta_{t_i} T'\delta_t\right)+\sum_{j=1}^k B_d\left(\prod_{i\neq j} T_i\delta_{t_i} (T_j\bullet_t T')\delta_{t_j}\right),\\
\phi(T\bullet T')&=a_d\bullet \prod_{i=1}^k \phi(T_i)\delta_{t_i}\phi(T')\delta_t
+\sum_{j=1}^k a_d\bullet \prod_{i\neq j} \phi(T_i)\delta_{t_i} (\phi(T_j\bullet_t T'))\delta_{t_j}\\
&=a_d\bullet \prod_{i=1}^k \phi(T_i)\delta_{t_i}\phi(T')\delta_t
+\sum_{j=1}^k a_d\bullet \prod_{i\neq j} \phi(T_i)\delta_{t_i} (\phi(T_j)\bullet_t \phi(T'))\delta_{t_j}\\
&=a_d\bullet \prod_{i=1}^k \phi(T_i)\delta_{t_i}\phi(T')\delta_t+a_d\bullet\left(\left(\prod_{i=1}^k 
\phi(T_i)\delta_{t_i}\right)\bullet \phi(T')\delta_t\right)\\
&=\left(a_d\bullet\prod_{i=1}^k \phi(T_i)\delta_{t_i} \right)\bullet \phi(T')\delta_t\\
&=\phi(T)\bullet_t \phi(T').
\end{align*}
So $\phi$ is a $\T$-multiple pre-Lie algebra morphism. \end{proof}

\subsection{Operad of typed trees}

We now describe an operad of typed trees, in the category of species. 
We refer to \cite{Dotsenko,Vallette,Markl} for notations and definitions on operads.

\begin{notation} Let $A$ be a finite set.  If $T\in \bfT_\T(A)$ and $a\in T$:
\begin{enumerate}
\item The subtrees formed by the connected components of the set of vertices, descendants of $a$ ($a$ excluded)
are denoted by $T_1^{(a)},\ldots,T_{n_a}^{(a)}$. 
The type of the edge from $a$ to the root of $T_i^{(a)}$ is denoted by $t_i$.
\item The tree formed by the vertices of $T$ which are not in $T_1^{(a)},\ldots,T_{n_a}^{(a)}$, 
at the exception of $a$, is denoted by $T_0^{(a)}$.
\end{enumerate}\end{notation}

\begin{prop}
For any nonempty finite set $A$, we denote by $\P_\T(A)$ the vector space generated by $\bfT_\T(A)$.
We define a composition $\circ$ on $\P_\T$ in the following way: for any $T\in \bfT_\T(A)$, 
$T'\in \bfT_\T(B)$ and $a\in A$,
\[T\circ_a T'=\sum_{v_1,\ldots,v_{n_a}\in V(T')}
\left(\ldots\left (\left(T_0^{(a)}\bullet_\lambda^{(t_0)} T'\right)\bullet_{v_1}^{(t_1)} T_1^{(a)}\right)\ldots\right)
\bullet_{v_{n_a}}^{(t_{n_a}^{(a)})} T_{n_a}^{(a)}.\]
With this composition, $\P_\T$ is an operad in the category of species.
\end{prop}
 
\begin{proof} Note that the tree 
$\left(\ldots \left(\left(T_0^{(a)}\bullet_\lambda^{(t_0)} T'\right)\bullet_{v_1}^{(t_1)} T_1^{(a)}\right)\ldots\right)
\bullet_{v_{n_a}}^{(t_{n_a}^{(a)})} T_{n_a}$,
which is shortly denote by $T\bullet_\lambda^{(v)} T'$, is obtained in the following process:
\begin{enumerate}
\item Delete the branches $T_1^{(a)},\ldots,T_{n_a}^{(a)}$ coming from $a$ in $T$. 
One obtains a tree $T''$, and $a$ is a leaf of $T''$.
\item Identify $a\in V(T'')$ with the root of $T'$.
\item Graft $T_1^{(a)}$ on $v_1$,$\ldots$, $T_{n_a}^{(a)}$ on $v_{n_a}$.
\end{enumerate}
This obviously does not depend on the choice of the indexation of $T_1^{(a)},\ldots,T_{n_a}^{(a)}$.\\

Let $T\in \bfT_\T(A)$, $T'\in \bfT_\T(B)$, $T''\in \bfT_\T(C)$.
\begin{itemize}
\item If $a',a''\in A$, with $a'\neq a''$, then:
\begin{align*}
(T\circ_{a'} T')\circ_{a''}T''&=\sum_{v'\in V(T')^{n_{a'}},v''\in V(T'')^{n_{a''}}}
\left(T\bullet_{a'}^{(v')}T'\right)\bullet_{a''}^{(v'')} T''\\
&=\sum_{v'\in V(T')^{n_{a'}},v''\in V(T'')^{n_{a''}}}\left(T\bullet_{a''}^{(v'')}T''\right)\bullet_{a'}^{(v')} T'\\
&=(T\circ_{a''} T'')\circ_{a'}T'.
\end{align*}
\item If $a'\in A$ and $b''\in B$, then:
\begin{align*}
(T\circ_{a'}T')\circ_{b''}T''&=\sum_{v'\in V(T')^{n_{a'}},v''\in V(T'')^{n_{b''}}}
\left(T\bullet_{a'}^{(v')}T'\right)\bullet_{b''}^{(v'')} T''\\
&=\sum_{v'\in V(T')^{n_{a'}},v''\in V(T'')^{n_{b''}}}T\bullet_{a'}^{(v')}\left(T'\bullet_{b''}^{(v'')} T''\right)\\
&=T\circ_{a'}(T'\circ_{b''}T'').
\end{align*}
\end{itemize} 
Moreover, $\tdun{$a$}\bullet_\lambda T=T$ for any tree $T$, and if $a\in V(T)$, $T\bullet_\lambda \tdun{$a$}T$.
So $\P_\bfT$ is indeed an operad in the category of species. \end{proof}

Consequently, the family $(\P_\T(n))_{n\geq 0}$ is a "classical" operad, which we denote by $\P_\T$.\\

\begin{example}
\begin{align*}
\tddeux{$1$}{$2$}{10}\circ_1\tddeux{$1$}{$2$}{2}&=\tdtroisun{$1$}{$3$}{$2$}{2}{10}
+\tdtroisdeux{$1$}{$2$}{$3$}{2}{10},&
\tddeux{$1$}{$2$}{2}\circ_2\tddeux{$1$}{$2$}{10}&=\tdtroisdeux{$1$}{$2$}{$3$}{2}{10}.
\end{align*}
\end{example}

\begin{remark} 
Another operad on typed trees is introduced in \cite{Dotsenko2}.
It is a typed version of the operad of nonassociative, permutative operad of \cite{Livernet},
and is different from ours.
\end{remark}

In the non-typed case, this theorem is proved in \cite{ChapotonLivernet}:

\begin{theo} \label{theo11}
The operad of $\T$-multiple pre-Lie algebras is isomorphic to $\P_\T$, 
via the isomorphism $\Phi$ sending, for any $t\in \T$, $\bullet_t$ to the tree
$\tddeux{$1$}{$2$}{10}$, where the edge is of type $t$.
\end{theo}

\begin{proof} The operad of $\T$-multiple pre-Lie algebras is generated by the binary elements $\bullet_t$,
$t\in \T$, with the relations
\begin{align*}
&\forall t,t'\in \T,& \bullet_{t'}\circ_2 \bullet_t-\bullet_t\circ_1\bullet_{t'}
&=(\bullet_t\circ_2\bullet_{t'}-\bullet_{t'}\circ_1 \bullet_t)^{(23)}.
\end{align*}
Firstly, if $t$ and $t'$ are elements of $\T$, symbolized by$\typ{10}$ and $\typ{2}$, by the preceding example:
\begin{align*}
\tddeux{$1$}{$2$}{10}\circ_1\tddeux{$1$}{$2$}{2}-\tddeux{$1$}{$2$}{2}\circ_2\tddeux{$1$}{$2$}{10}
&=\tdtroisun{$1$}{$3$}{$2$}{2}{10}
=\left(\tdtroisun{$1$}{$3$}{$2$}{10}{2}\right)^{(23)}
=\left(\tddeux{$1$}{$2$}{2}\circ_1\tddeux{$1$}{$2$}{10}-\tddeux{$1$}{$2$}{10}\circ_2\tddeux{$1$}{$2$}{2}
\right)^{(23)}.
\end{align*}
So the morphism $\phi$ exists. Let us prove that it is surjective:
let $T\in \bfT_\T(n)$, we show that it belongs to $\im(\Phi)$ by induction on $n$. It is obvious if $n=1$ or $n=2$.
Let us assume the result at all ranks $<n$. Up to a reindexation we assume that:
\[T=B_1(T_1\delta_{t_1}\ldots T_k\delta_{t_k}),\]
 where for any $1\leq i<j\leq k$, if $x\in V(T_i)$ and $y\in V(T_j)$, then $x<y$.
We denote by $T'_i$ the standardization of $T_i$. By the induction hypothesis on $n$, $T'_i\in \im(\Phi)$ for all $i$.
We proceed by induction on $k$. The type $t_k$ will be represented in red. If $k=1$, then:
\[T=\tddeux{$1$}{$2$}{10}\circ_2 T_1\in \im(\Phi).\]
Let us assume the result at rank $k-1$. We put $T'=B_1(T_1\delta_{t_1}\ldots T_{k-1}\delta_{t_{k-1}})$. 
By the induction hypothesis on $n$, $T'\in \im(\Phi)$. Then:
\[\tddeux{$1$}{$2$}{10}\circ_1 T'=T+x,\]
where $x$ is a sum of trees with $n$ vertices, such that the fertility of the root is $k-1$. 
Hence, $x\in \im(\Phi)$, so $T\in \im(\Phi)$.\\

Let $\D$ be a set. The morphism $\phi$ implies that the free $\P_\bfT$-algebra generated by $\D$, 
that is to say $\g_{\D,\T}$, inherits
a $\T$-multiple pre-Lie algebra structure defined by:
\begin{align*}
&\forall x,y\in \g_{\D,\T},\: \forall \typ{2}\in \T,& x\circ_{\typ{2}} y=\tddeux{$1$}{$2$}{2}\cdot(x\otimes y),
\end{align*}
where $\cdot$ is the $\P_\T$-algebra structure of $\g_{\D,\T}$.
For any trees $T$, $T'$ in $\bfT_{\D,\T}$, by definition of the operadic composition of $\P_\T$:
\[T\circ_t T'=\sum_{v\in V(T)} T\bullet_t^{(v)} T',\]
so $\circ_t=\bullet_t$ for any $t$. As $(\g_{\D,\T},(\bullet_t)_{t\in \T})$ 
is the free $\T$-multiple pre-Lie algebra generated by $\D$,
$\Phi$ is an isomorphism. \end{proof}

\begin{remark}
Let us assume that $\T$ is finite, of cardinality $T$. Then the components of $\P_\T$ are finite-dimensional. 
As the number of rooted trees which vertices are the elements of $[n]$ is $n^{n-1}$, for any $n\geq 0$
the dimension of $\P_\T(n)$ is $T^{n-1}n^{n-1}$, and the formal series of $\P_\T$ is:
\[f_T(X)=\sum_{n\geq 1} \frac{\dim(\P_\T(n))}{n!} X^n=\sum_{n\geq 1} \frac{(Tn)^{n-1}}{n!} X^n=\frac{f_1(TX)}{T}.\]
\end{remark}

\subsection{Koszul dual operad}

If $\T$ is finite, then $\P_\T$ is a quadratic operad. Its Koszul dual can be directly computed:

\begin{prop} \label{prop12}
The Koszul dual operad $\P_\T^!$ of $\P_\T$ is generated by $\diamond_t$, $t\in \T$, with the relations:
\begin{align*}
&\forall t,t'\in \T,& \diamond_{t'}\circ_1 \diamond_t&=\diamond_t\circ_2 \diamond_{t'},
&\diamond_{t'}\circ_1 \diamond_t&=(\diamond_t\circ _1 \diamond_{t'})^{(23)}.
\end{align*}
The algebras on $\P_\T^!$ are called $\T$-multiple permutative algebras. 
Such an algebra $A$ is given bilinear products $\diamond_t$, $t\in \T$, such that:
\begin{align*}
&\forall x,y,z\in A,& (x\diamond_t y)\diamond_{t'} z&=x\diamond_t (y\diamond_{t'} z),\\
&&(x\diamond_t y)\diamond_{t'} z&=(x\diamond_{t'} z)\diamond_t y.
\end{align*}
In particular, for any $t$, $\diamond_t$ is a permutative product. 
\end{prop}

Of course, the definition of $\T$-multiple permutative algebras makes sense even if $\T$ is infinite.
Permutative algebras are introduced in \cite{Chapoton}. If $A$ is a $\T$-multiple permutative algebra,
then for any $(\lambda_t)_{t\in \T}\in \K^{(\T)}$, $\diamond_a=\sum \lambda_t\diamond_t$ is a a permutative product on $A$.

\begin{prop} \label{prop13}
Let $V$ be a vector space. Then $V\otimes S\left(V^{\oplus\T}\right)$ is given a $\T$-multiple permutative algebra structure:
\begin{align*}
&\forall t\in \T,\: v,v'\in V, w,w'\in S\left(V^{\oplus\T}\right),&
(v\otimes w)\diamond_t (v'\otimes w')&=v \otimes ww' (v'\delta_t). 
\end{align*} 
This $\T$-multiple permutative algebra is denoted by $P_\T(V)$. 
For any $\T$-multiple permutative algebra $V$ and any linear map $\phi:V\longrightarrow A$,
there exists a unique morphism $\Phi:P_\T(V)\longrightarrow A$ such that for any $v\in V$, $\Phi(v\otimes 1)=\phi(v)$.
\end{prop}

\begin{proof}
Let $t$, $t'\in \T$, $v$, $v$, $v'' \in V$,
$w$, $w'$, $w''\in S\left(V^{\oplus\T}\right)$. 
\begin{align*}
(v\otimes w \diamond_t v'\otimes w')\diamond_{t'} v''\otimes w''
&=v\otimes w \diamond_t (v'\otimes w'\diamond_{t'} v''\otimes w'')\\
&=(v\otimes w \diamond_t v'\diamond_{t'} v''\otimes w'')\otimes w'\\
&=v\otimes ww'w'' (v'\delta_t)(v''\delta_{t'}),
\end{align*}
so $P_\T(V)$ is $\T$-multiple permutative. \\

\textit{Existence of $\Phi$}. Let $t_1,\ldots,t_k\in \T$, $v,v_1,\ldots,v_k\in V$. We inductively
define $\Phi(v\otimes (v_1\delta_{t_1})\ldots (v_k\delta_{t_k}))$ by:
\begin{align*}
\Phi(v\otimes 1)&=\phi(v),\\
\Phi(v\otimes (v_1\delta_{t_1})\ldots (v_k\delta_{t_k}))
&=\Phi(v\otimes (v_1\delta_{t_1})\ldots (v_{k-1}\delta_{t_{k-1}})) \diamond_{t_k} \phi(v_k) \mbox{ if }k\geq 1.
\end{align*}
Let us prove that this does not depend on the order chosen on the factors $v_i\delta_{t_i}$ by induction on $k$.
If $k=0$ or $1$, there is nothing to prove. Otherwise, if $i<k$:
\begin{align*}
&\Phi(v\otimes (v_1\delta_{t_1})\ldots (v_{i-1}\delta_{t_{i-1}})(v_{i+1}\delta_{t_{i+1}})\ldots
(v_k\delta_{t_k}))\diamond_{t_i} \phi(v_i)\\
&=(\Phi(v\otimes (v_1\delta_{t_1})\ldots (v_{i-1}\delta_{t_{i-1}})(v_{i+1}\delta_{t_{i+1}})\ldots
(v_{k-1}\delta_{t_{k-1}}))\diamond_{t_k} \phi(v_k))\diamond_{t_i} \phi(v_i)\\
&=(\Phi(v\otimes (v_1\delta_{t_1})\ldots (v_{i-1}\delta_{t_{i-1}})(v_{i+1}\delta_{t_{i+1}})\ldots
(v_{k-1}\delta_{t_{k-1}}))\diamond_{t_i} \phi(v_i))\diamond_{t_k} \phi(v_k)\\
&=\Phi(v\otimes (v_1\delta_{t_1})\ldots (v_{k-1}\delta_{t_{k-1}}))\diamond_{t_k} \phi(v_k)\\
&=\Phi(v\otimes (v_1\delta_{t_1})\ldots (v_k\delta_{t_k})).
\end{align*}
So $\Phi$ is well-defined. Let us prove that $\Phi$ is a $\T$-multiple permutative algebra morphism.
Let $v,v'\in V$, $w$, $w'=(v_1\delta_{t_1})\ldots (v_k\delta_{t_k}) \in S\left(V^{\oplus\T}\right)$, and $t\in \T$.
Let us prove that $\Phi(v\otimes w \diamond_t v'\otimes w'')=\Phi(v\otimes w)\diamond_t \Phi(v'\otimes w')$
by induction on $k$. If $k=0$:
\begin{align*}
\Phi(v\otimes w \diamond_t v'\otimes 1)&=\Phi(v\otimes w(v'\delta_t))\\
&=\Phi(v\otimes w)\diamond_t \phi(v')\\
&=\Phi(v\otimes w)\diamond_t \Phi(v'\otimes 1).
\end{align*} 
Otherwise, we put $w''=(v_1\delta_{t_1})\ldots (v_{k-1}\delta_{t_{k-1}})$. Then:
\begin{align*}
\Phi(v\otimes w \diamond_t v'\otimes w')&=\Phi(v\otimes ww'' (v'\delta_t)(v_k \delta_{t_k}))\\
&=\Phi(v\otimes ww'' (v'\delta_t))\diamond_{t_k} \phi(v_k)\\
&=\Phi(v\otimes w \diamond_t v'\otimes w'')\diamond_{t_k} \phi(v_k)\\
&=(\Phi(v\otimes w)\diamond_t \Phi(v'\otimes w''))\diamond_{t_k} \phi(v_k)\\
&=\Phi(v\otimes w)\diamond_t (\Phi(v'\otimes w'')\diamond_{t_k} \phi(v_k))\\
&=\Phi(v\otimes w')\diamond_t \Phi(v'\otimes w').
\end{align*}
So $\Phi$ is a $\T$-multiple permutative algebra morphism.\\

\textit{Uniqueness}. For any $v,v_1,\ldots,v_k \in V$, $t_1,\ldots,t_k \in \T$:
\[v\otimes (v_1\delta_{t_1})\ldots (v_k\delta_{t_k})
=(v\otimes (v_1\delta_{t_1})\ldots (v_{k-1}\delta_{t_{k-1}}))\diamond_{t_k} v_k.\]
It is then easy to prove that $P_\T(V)$ is generated by $V\otimes 1$ as a $\T$-multiple permutative algebra.
Consequently, $\Phi$ is unique. \end{proof}

\begin{remark}\begin{enumerate}
\item We proved that $P_\T(V)$ is freely generated by $V$, identified with $V\otimes 1$. Consequently,
$\P_\T^!(n)$ has the same dimension as the multilinear component of $V\otimes S\left(V^{\oplus\T}\right)$ with
$V=Vect(X_1,\ldots,X_n)$, that is to say:
\[Vect(X_i\otimes (X_1\delta_{t_1})\ldots (X_{i-1}\delta_{t_{i-1}})(X_{i+1}\delta_{t_{i+1}})
\ldots (X_n\delta_{t_n}),\: 1\leq i \leq n, t_j\in \T),\]
so:
\[\dim(\P_\T^!(n))=n T^{n-1}.\]
The formal series of $\P_\T^!$ is:
\[f^!_T(X)=\sum_{n\geq 1} \frac{\dim(\P_\T^!(n))}{n!} X^n
=\sum_{n\geq 1} \frac{T^{n-1}}{(n-1)!} X^n=X\mathrm{exp}(TX)=\frac{f^!_1(TX)}{T}.\]
\item It is possible to prove that $\P_\T^!$ is a Koszul operad (and, hence, $\P_\T$ too) using the rewriting method of \cite{Vallette}.
\end{enumerate}\end{remark}

\section{Structure of the pre-Lie products}

\subsection{A nonassociative permutative coproduct}

\begin{prop} \label{prop14}
For all $t\in \T$, we define a coproduct $\rho_t:\g_{\D,\T}\longrightarrow \g_{\D,\T}^{\otimes 2}$ by:
\begin{align*}
&\forall T=B_d\left(\prod_{i\in [k]} T_i \delta_{t_i}\right) \in \bfT_{\D,\T},&
\rho_t(T)&=\sum_{j\in [k]} B_d\left(\prod_{i\in [k],\:i\neq j} T_i \delta_{t_i}\right) \otimes T_j \delta_{t,t_j}.
\end{align*}
Then:
\begin{enumerate}
\item For all $t,t'\in \T$, $(\rho_t\otimes \id)\circ \rho_{t'}=((\rho_{t'}\otimes \id)\circ \rho_t)^{(23)}$.
\item For any $x,y\in \g_{\D,\T}$, $t,t'\in \T$, with Sweedler's notations 
$\rho_t(x)=\sum x^{(1)_t}\otimes x^{(2)_t}$,
\[\rho_t(x\bullet_{t'} y)=\delta_{t,t'} x\otimes y+\sum x^{(1)_t}\bullet_{t'} y\otimes x^{(2)_t}
+\sum x^{(1)_t}\otimes x^{(2)_t}\bullet_{t'} y.\]
\item For any $\mu=(\mu_t)_{t\in \T}\in \K^\T$, we put:
\[\rho_\mu=\sum_{t\in \T} \mu_t \rho_t:\g_{\D,\T}\longrightarrow \g_{\D,\T}^{\otimes 2}.\]
This makes sense, as any tree in $\bfT_{\D,\T}$ does not vanish only under a finite number of $\rho_t$.
Then $\rho_\mu$ is a nonassociative permutative (NAP) coproduct; 
for any $x,y\in \g_{\D,\T}$, by the second point, using Sweeder's notation for $\rho_\mu$:
\begin{align*}
\rho_\mu(x\bullet_\lambda y)&=\left(\sum_{t\in \T} \lambda_t\mu_t\right) x\otimes y
+\sum x^{(1)_\mu}\bullet_\lambda y\otimes x^{(2)_\mu}
+\sum x^{(1)_\mu}\otimes x^{(2)_\mu}\bullet_\lambda y.
\end{align*}
In particular, if $\displaystyle \sum_{t\in \T} \lambda_t\mu_t=1$,
$(\g_{\D,\T},\bullet_\lambda,\rho_\mu)$ is a NAP pre-Lie bialgebra in the sense of \cite{Livernet}.
\end{enumerate}\end{prop}

\begin{proof} 1. For any tree $T$:
\begin{align*}
(\rho_t\otimes \id)\circ \rho_{t'}(T)&=\sum_{p,q\in [k], p\neq q} B_d\left(\prod_{i\in [k], i\neq p,q}T_i\delta_{t_i}
\right)\otimes T_p \delta_{t_p,t}\otimes T_q\delta_{t_q,t'},
\end{align*}
which implies the result. \\

2. For any tree $T,T'$:
\begin{align*}
\rho_t(T\bullet_{t'} T')&=\rho_t\left(B_d\left(\prod_{i\in [k]}T_i\delta_{t_i} T'\delta_{t'}\right)
+\sum_{i\in [k]} B_d\left(\prod_{j\in [k], j\neq i}T_j \delta_{t_j} (T_i\bullet_{t'}T')\delta_{t_i}\right)\right)\\
&=B_d\left(\prod_{i\in [k]}T_i\delta_{t_i}\right)\otimes T'\delta_{t,t'}+\sum_{i\in [k]}
B_d\left(\prod_{j\in [k],j\neq i}T_j\delta_{t_j} T'\delta_{t'}\right)\otimes T_i\delta_{t_i,t}\\
&+\sum_{i\in [k]}B_d\left(\prod_{j\in [k], j\neq i}T_j \delta_{t_j} \right)\otimes (T_i\bullet_{t'}T')\delta_{t_i,t'}\\
&+\sum_{i\neq j \in [k]}B_d\left(\prod_{p\in [k],p\neq i,j} T_p \delta_{t_p} (T_j\bullet_{t'} T')\delta_{t_j}\right)\otimes T_i \delta_{t_i,t}\\
&=T\otimes T'\delta_{t,t'}+\sum_{i\in [k]} B_d\left(\prod_{j\in [k],j\neq i}T_j\delta_{t_j}\right)\bullet_{t'} T'\otimes T_i \delta_{t_i,t}\\
&+\sum_{i\in [k]} B_d\left(\prod_{j\in [k],j\neq i}T_j\delta_{t_j}\right)\otimes T_i \bullet_{t'} T'\delta_{t_i,t}\\
&=T\otimes T'\delta_{t,t'}+T^{(1)_t}\bullet_{t'}T'\otimes T^{(2)_t}+T^{(1)_t}\otimes T^{(2)_t}\bullet_{t'}T'.
\end{align*}

3. Obtained by summation. \end{proof}

\begin{cor} \label{cor15}
If $\lambda \in \K^{(\T)}$ is nonzero, let us choose $t_0\in \T$ such that $\lambda_{t_0}\neq 0$.
The pre-Lie algebra $(\g_{\D,\T}, \bullet_\lambda)$ is freely generated by the set 
$\bfT_{\D,\T}^{(t_0)}$ of $\T$-typed $\D$-decorated trees $T$ such that
there is no edge outgoing the root of $T$ of type $t_0$.
\end{cor}

\begin{proof} For any tree $T$, we denote by $\alpha_T$ the number of edges outgoing the root of $T$ of type $T_0$.
Our aim is to prove that $(\g_{\D,\T},\bullet_\lambda)$ is freely generated by the trees $T$ such that $\alpha_T=0$.
We define a family of scalar $b$ by:
\begin{align*}
&\forall t\in \T,&\mu_t&=\begin{cases}
0\mbox{ if }t\neq t_0,\\
\displaystyle \frac{1}{\lambda_{t_0}} \mbox{ if }t=t_0.
\end{cases}\end{align*}
Note that $\displaystyle \rho_\mu=\frac{1}{\lambda_{t_0}} \rho_{t_0}$.
By proposition \ref{prop14}, $(\g_{\D,\T},\bullet_\lambda,\rho_\mu)$ is a NAP pre-Lie bialgebra, 
so by Livernet's rigidity theorem \cite{Livernet}, it is freely generated by $\ker(\rho_\mu)=\ker(\rho_{t_0})$.
Obviously, if $\alpha_T=0$, $T\in \ker(\rho_{t_0})$. 
Let us consider $\displaystyle x=\sum_{T\in \bfT_{\D,\T}} x_T T\in \ker(\rho_{t_0})$. We consider the map:
\[\Upsilon:\left\{\begin{array}{rcl}
\g_{\D,\T}\otimes \g_{\D,\T}&\longrightarrow&\g_{\D,\T}\\
T\otimes T'&\longrightarrow&T\bullet_{t_0}^{\mathrm{root}(T)}T'.
\end{array}
\right.\]
By definition of $\rho_{t_0}$, for any tree $T$, $\Upsilon \circ \rho_{t_0}(T)=\alpha_T T$. Consequently:
\[0=\Upsilon \circ \rho_{t_0}(x)=\sum_{T\in \bfT_{\D,\T}} x_T\alpha_TT.\]
So if $\alpha_T\neq 0$, $x_T=0$, and $x$ is a linear span of trees such that $\alpha_T=0$ : 
the set of trees $T$ such that $\alpha_T=0$ is a basis of $\ker(\rho_{t_0})$. \end{proof}

If $|\D|=D$ and $|\T|=T$, the number of elements of $\bfT_{\D,\T}^{(t_0)}$  
of degree $n$ is denoted by $t'_{D,T}(n)$;
it does not depend on $t_0$. By direct computations:
\begin{align*}
t'_{D,T}(1)&=D,\\
t'_{D,T}(2)&=D^2(T-1),\\
t'_{D,T}(3)&=\frac{D^2(T-1)(3DT-D+1)}{2},\\
t'_{D,T}(4)&=\frac{D^2(T-1)(16D^2T^2-8D^2T+D^2+6DT-3D+2)}{6}.
\end{align*}
In the particular case $D=1$, $T=2$, we recover sequence A005750 of the OEIS.

\subsection{Pre-Lie algebra morphisms}

\begin{notation} Let $\T$ and $\T'$ be two sets of types.
We denote by $\M_{\T,\T'}(\K)$ the space of matrices $M=(m_{t,t'})_{(t,t')\in \T\times \T'}$, 
such that for any $t'\in \T'$,
$(m_{t,t'})_{t\in \T} \in \K^{(\T)}$. 
If $\T=\T'$, we shall simply write $\M_\T(\K)$.
If $M\in \M_{\T,\T'}(\K)$ and $M'\in M_{\T',\T''}(\K)$, then:
\[MM'=\left(\sum_{t'\in \T'} m_{t,t'}m'_{t',t''}\right)_{(t,t'')\in \T\times \T''} \in \M_{\T,\T''}(\K).\]
If $\lambda\in \K^{(\T')}$ and $\mu\in \K^\T$, then:
\begin{align*}
M\lambda&=\left(\sum_{t'\in \T'}m_{t,t'}\lambda_{t'}\right)_{t\in \T}\in \K^{(\T)},&
M^\top \mu&=\left(\sum_{t\in \T}m_{t,t'}\mu_t\right)_{t'\in \T'}\in \K^{\T'}.
\end{align*}
In particular, $\M_\T(\K)$ is an algebra, acting on $\K^{(\T)}$ on the left and on $\K^\T$ on the right.
\end{notation}

\begin{defi}
Let $M\in \M_{\T,\T'}(\K)$. We define a map $\Phi_M:\h_{\D,\T'}\longrightarrow \h_{\D,\T}$, 
sending $F\in \bfF_{\D,\T}$ to the forest obtained by replacing $\type(e)$ by $\displaystyle \sum_{t\in \T} 
m_{t,\type(e)} t$ for any $e\in E(F)$, $F$ being considered as linear in any of its edges.
The restriction of $\Phi_M$ to $\g_{\D,\T'}$ is denoted by $\phi_M:\g_{\D,\T'}\longrightarrow \g_{\D,\T}$.
\end{defi}

\begin{example}
If $\T$ contains two elements, the first one represented by $\typ{10}$ and the second one by $\typ{2}$, 
if $M=\begin{pmatrix}\alpha&\beta\\\gamma&\delta\end{pmatrix}$, for any $x,y,z\in \D$:
\begin{align*}
\phi_M\left(\tddeux{$x$}{$y$}{10}\right)&=\alpha\tddeux{$x$}{$y$}{10}+\gamma\tddeux{$x$}{$y$}{2},\\
\phi_M\left(\tddeux{$x$}{$y$}{2}\right)&=\beta\tddeux{$x$}{$y$}{10}+\delta\tddeux{$x$}{$y$}{2},\\
\phi_M\left(\tdtroisun{$x$}{$z$}{$y$}{10}{2}\right)&=\alpha\beta\tdtroisun{$x$}{$z$}{$y$}{10}{10}
+\alpha\delta\tdtroisun{$x$}{$z$}{$y$}{10}{2}+\beta\gamma\tdtroisun{$x$}{$z$}{$y$}{2}{10}
+\gamma\delta\tdtroisun{$x$}{$z$}{$y$}{2}{2}.
\end{align*}\end{example}

\begin{remark}
For any $M\in \M_{\T,\T'}(\K)$, $M'\in \M_{\T',\T''}(\K)$, $\Phi_M\circ \Phi_{M'}=\Phi_{MM'}$.
\end{remark}

\begin{prop}
Let $\lambda\in \K^{(\T)}$, $\mu\in \K^\T$ and $M\in \M_{\T,\T'}(\K)$. 
Then $\phi_M$ is a pre-Lie morphism from $(\g_{\D,\T'},\bullet_\lambda)$ to $(\g_{\D,\T},\bullet_{M\lambda})$
and  a NAP coalgebra morphism from $(\g_{\D,\T'},\rho_{M^\top \mu})$ to $(\g_{\D,\T},\rho_\mu)$.
\end{prop}

\begin{proof} Let $T,T'\in \bfT_{\D,\T}$. For any $t\in \T$, for any $v\in V(T)$:
\begin{align*}
\phi_M\left(T\bullet_t^{(v)} T'\right)=\sum_{t'\in \T} m_{t',t} \phi_M(T)\bullet_{t'} \phi_M(T'),
\end{align*}
so:
\begin{align*}
\phi_M(T\bullet_\lambda T')&=\sum_{t,t'\in \T}m_{t',t}\lambda_t \phi_M(T)\bullet_{t'}\phi_M(T')=\phi_M(T)\bullet_{M\lambda}\phi_M(T').
\end{align*}
We proved that $\phi_M$ is a pre-Lie algebra morphism from $(\g_{\D,\T'},\bullet_\lambda)$ to $(\g_{\D,\T},\bullet_{M\lambda})$.\\

For any $T \in \bfT_{\D,\T}$:
\begin{align*}
\rho_t\circ \phi_M(T)&=\sum_{t'\in \T} m_{t,t'} (\phi_M\otimes \phi_M)\circ \rho_{t'}(T),
\end{align*}
so:
\begin{align*}
\rho_\mu\circ \phi_M(T)&=\sum_{t,t'\in \T} m_{t,t'} \mu_t (\phi_M\otimes \phi_M)\circ \rho_{t'}(T)
=(\phi_M\otimes \phi_M)\circ \rho_{M^\top \mu}(T).
\end{align*}
So $\phi_M:(\g_{\D,\T'},\rho_{M^\top \mu})\longrightarrow(\g_{\D,\T},\rho_\mu)$ is a NAP coalgebra morphism.\end{proof}

\begin{cor} \label{cor18}
For any $\lambda \in\in \K^{(\T)}$ and $\mu\in \K^\T$, 
such that $\displaystyle \sum_{t\in \T} \lambda_t\mu_t=1$, for any $t_0\in \T$,
the NAP pre-Lie bialgebras $(\g_{\D,\T},\bullet_\lambda,\rho_\mu)$ and $(\g_{\D,\T},\bullet_{t_0},\rho_{t_0})$ are isomorphic.
\end{cor}

\begin{proof} Let us denote by $\lambda^{(0)}$ the element of $\K^{(\T)}$ defined by:
\[\lambda^{(0)}_t=\delta_{t,t_0}.\]
Note that for any $M\in \M_\T(\K)$, invertible,  
$\phi_M:(\g_{\D,\T},\bullet_{\lambda^{(0)}},\rho_{M^\top \mu})\longrightarrow 
(\g_{\D,\T},\bullet_{M\lambda^{(0)}},\rho_\mu)$ is an isomorphism.
In particular, for a well-chosen $M$, $M\lambda^{(0)}=\lambda$; we can assume that 
$\lambda=\lambda^{(0)}$ without loss of generality.
Then, by hypothesis, $\mu_{t_0}=1$. We define a matrix $M\in \M_\T(\K)$ in the following way:
\[m_{t,t'}=\begin{cases}
\delta_{t,t_0} \mbox{ if }t'=t_0,\\
\delta_{t,t'}-\mu_{t'} \delta_{t,t_0} \mbox{ otherwise}.
\end{cases}\]
Then $M$ is invertible. Moreover, $M\lambda^{(0)}=\lambda^{(0)}$ and $M^\top \mu=\lambda^{(0)}$. So $\phi_M$
is an isomorphism from $(\g_{\D,\T},\bullet_{\lambda^{(0)}},\rho_{\lambda^{(0)}})$
to $(\g_{\D,\T},\bullet_\lambda,\rho_\mu)$. \end{proof}

\begin{prop} \label{prop19}
Let $\lambda\in \K^{(\T)}$, and $t_0\in \T$. 
We define a pre-Lie algebra morphism $\psi_{t_0}:
(\g_{\bfT_{\D,\T}^{(t_0)}},\bullet)\longrightarrow (\g_{\D,\T},\bullet_\lambda)$,
sending $\tdun{$T$}$ to $T$ for any $T\in \bfT_{\D,\T}^{(t_0)}$. Then $\psi_{t_0}$ 
is a pre-Lie algebra isomorphism if, and only if, $\lambda_{t_0}\neq 0$.
\end{prop}

\begin{proof} If $\lambda_{t_0}\neq 0$, then by corollary \ref{cor15}, $(\g_{\D,\T},\bullet_\lambda)$ is freely generated by $\bfT_{\D,\T}^{(t_0)}$, so
$\psi_{t_0}$ is an isomorphism.
If $\lambda_{t_0}=0$, then it is not difficult to show that any tree $T$ with two vertices, with its unique edge of type $t_0$, does not belong to $\im(\psi_{t_0})$.
\end{proof}

\section{Hopf algebraic structures}

We here fix $\lambda\in \K^{(\T)}$.

\subsection{Enveloping algebra of $\g_{\D,\T}$}

Using again the Guin-Oudom construction, we obtain the enveloping algebra of $(\g_{\D,\T},\bullet_\lambda)$.
We first identify the symmetric coalgebra $S(\g_{\D,\T})$ with the vector space generated by $\bfF_{\D,\T}$, 
which we denote by $\h_{\D,\T}$.
Its product $m$ is given by disjoint union of forests, its coproduct by:
\begin{align*}
&\forall T_1,\ldots,T_k\in \bfT_{\D,\T},
&\Delta(T_1\ldots T_n)=\sum_{I\subseteq [n]} \prod_{i\in I}T_i\otimes \prod_{i\notin I} T_i.
\end{align*}
We denote by $\bullet_\lambda$ the Guin-Oudom extension of $\bullet_\lambda$ to $\h_{\D,\T}$ and $\star_\lambda$ the associated associative product. 

\begin{theo}
For any $F\in \bfF_{\D,\T}$, $T_1,\ldots,T_n \in \bfT_{\D,\T}$:
\begin{align*}
F\bullet_\lambda T_1\ldots T_n&=\sum_{\substack{v_1,\ldots,v_n \in V(F),\\ t_1,\ldots,t_n \in \T}}
\left(\prod_{i\in [n]} \lambda_{t_i}\right) \left(\ldots\left(F \bullet_{t_1}^{(v_1)}T_1\right)\ldots\right)\bullet_{t_n}^{(v_n)}T_n,\\
F\star_\lambda T_1\ldots T_n&=\sum_{I\subseteq [n]}  \left(F\bullet_\lambda 
\prod_{i\in I}T_i\right)\prod_{i\notin I} T_i.
\end{align*}
The Hopf algebra $(\h_{\D,\T},\star_\lambda,\Delta)$ is denoted by $\h_{\D,\T}^{GL_\lambda}$.
Moreover, for any $M\in \M_{\T,\T'}(\K)$, for any $\lambda\in \K^{(\T')}$,
$\Phi_M$ is a Hopf algebra morphism from $\h_{\D,\T'}^{GL_\lambda}$ to $\h_{\D,\T}^{GL_{M\lambda}}$.
The extension of $\psi_{t_0}$ as a Hopf algebra morphism from $\h_{\bfT_{\D,\T}^{(t_0)}}^{GL}$ to 
$\h_{\D,\T}^{GL_\lambda}$
is denoted by $\Psi_{t_0}$; it is an isomorphism if, and only if, $\lambda_{t_0}\neq 0$. 
\end{theo}

In particular, if $\T=\{t\}$ and $\lambda_t=1$, we recover the Grossman-Larson Hopf algebra \cite{GrossmanLarson}.

\subsection{Dual construction}

\begin{prop}
Let $T\in \bfT_{\D,\T}$. \begin{enumerate}
\item A cut $c$ of $T$ is a nonempty subset of $E(T)$; it is said to be admissible 
if any path in the tree from the root to a leaf
meets at most one edge in $c$. The set of admissible cuts of $T$ is denoted by $\adm(T)$.
\item If $c$ is admissible, one of the connected components of $T\setminus c$ contains the root of $c$:
we denote it by $R^c(T)$. The product of the other connected components of $T\setminus c$ is denoted by $P^c(T)$.  
\end{enumerate}
Let $\lambda\in \K^\T$. We define a multiplicative coproduct $\Delta^{CK_\lambda}$ on the algebra $(\h_{\D,\T},m)$ by:
\begin{align*}
&\forall T\in \bfT_{\D,\T},& \Delta^{CK_\lambda}(T)&=T\otimes 1+1\otimes T+\sum_{c\in \adm(T)} 
\left(\prod_{e\in c}\lambda_{\type(e)}\right) R^c(T)\otimes P^c(T). 
\end{align*}
Then $(\h_{\D,\T},m,\Delta^{CK_\lambda})$ is a Hopf algebra, which we denote by $\h_{\D,\T}^{CK_\lambda}$.
\end{prop}

\begin{proof} We first assume that $\lambda\in \K^{(\T)}$. 
Let us define a nondegenerate pairing $\langle-,-\rangle$ on $\h_{\D,\T}$ by:
\begin{align*}
&\forall F,F'\in \bfF_{\D,\T},&\langle F,F'\rangle&=\delta_{F,F'}s_F,
\end{align*}
where $s_F$ is the number of symmetries of $F$.
Let us consider three forests $F,F',F''$. We put:
\begin{align*}
F&=\prod_{T\in \bfT_{\D,\T}} T^{\lambda_t},&F'&=\prod_{T\in \bfT_{\D,\T}} T^{a'_T},&F''&=\prod_{T\in \bfT_{\D,\T}} T^{a''_T}.
\end{align*}
Then:
\begin{align*}
\langle \Delta(F),F'\otimes F''\rangle&=\sum_{a=b+c} \prod_{T\in\bfT_{\D,\T}}\frac{\lambda_t!}{\mu_t!c_T!}
\langle \prod_{T\in \bfT_{\D,\T}} T^{\mu_t},F'\rangle \langle \prod_{T\in \bfT_{\D,\T}} T^{c_T},F''\rangle\\
 &=\sum_{a=b+c}\delta_{b,a'}\delta_{c,a''}\frac{\lambda_t!}{a'_T!a''_T!} s_{F'}s_{F''}\\
 &=\delta_{a,a'+a''}\frac{\lambda_t!}{a'_T!a''_T!} a'_T!a''_T! s_{T}^{a'_T+a''_T}\\
 &=\delta_{a,a'+a''}\lambda_t!s_{T}^{\lambda_t}\\
 &=\delta_{F,F'F''}s_F\\
 &=\langle F,F'F''\rangle.
\end{align*}
Therefore:
\begin{align*}
&\forall x,y,z\in \h_{\D,\T},&\langle \Delta(x),y\otimes z\rangle=\langle x,yz\rangle.
\end{align*}

Let $F,G$ be two forests and $T$ be a tree. Observe that if $F$ is a forest with at least two trees,
then $F\star_\lambda G$ does not contain any tree, so $\langle F\star_\lambda G,T\rangle=0$.
If $F=1$, then $\langle F\star_\lambda G,T\rangle\neq 0$ if, and only if, $G=T$; moreover, $\langle 1\star_\lambda T,T\rangle=1$.
If $F$ is a tree, then:
\[\langle F\star_\lambda G,T\rangle=\langle F\bullet_\lambda G,T\rangle.\]
Moreover, if $F=B_d(F')$ and $G=T_1\ldots T_k$:
\[F\bullet_\lambda G=\sum_{I\subseteq [k]}\: \sum_{(t_i)\in \T^k}\:
\left(\prod_{i\in [k]} \lambda_{t_i}\right) B_d\left(\prod_{i\in I} T_i\delta_{t_i} 
F'\bullet \prod_{i\notin I} T_i\delta_{t_i}\right),\]
where $\bullet$ is the pre-Lie product on $\g_{\D,\T}^\T$ induced by the $\T$-multiplie pre-Lie structure.
Consequently, we can inductively define a coproduct $\Delta^{CK_\lambda}:\h_{\D,\T}\longrightarrow 
\h_{\D,\T}\otimes \h_{\D,\T}$,
multiplicative for the product $m$, such that, if we denote for any tree $T$, 
$\overline{\Delta}_{CK}(T)=\Delta(T)-1\otimes T$,
for any tree $T=B_d(T_1\delta_{t_1}\ldots T_k\delta_{t_k})$:
\begin{align}
\label{EQ3} \overline{\Delta}^{CK}_\lambda(T)&=(B_d\otimes \id)\left(\prod_{i\in [k]}
(\overline{\Delta}_\lambda^{CK}(T_i)\delta_{t_i}\otimes 1+\lambda_{t_i}1\otimes T_i)\right).
\end{align} 
Then, for any $x,y,z\in \h_{\D,\T}$:
\[\langle x\star_\lambda y,z\rangle=\langle x\otimes y,\Delta^{CK_\lambda}(z)\rangle.\]
A quite easy induction on the number of vertices of trees proves that this coproduct is indeed 
the one we define in the statement of the proposition. As $\langle-,-\rangle$ is nondegenerate, 
$(\h_{\D,\T},m,\Delta^{CK_\lambda})$ is a Hopf algebra, dual to $\h_{\D,\T}^{GL_\lambda}$.\\

In the general case, for any $x\in \h_{\D,\T}$, there exists a finite subset $\T'$ of $\T$ such that 
$x\in \h_{\D,\T'}$. Putting $\lambda'=\lambda_{\mid \T'}$, $\lambda'\in \K^{\T'}=\K^{(\T')}$, so:
\begin{align*}
(\Delta^{CK_\lambda}\otimes \id)\circ \Delta^{CK_\lambda}(x)
&=(\Delta^{CK_{\lambda'}}\otimes \id)\circ \Delta^{CK_{\lambda'}}(x)\\
&=(\id \otimes \Delta^{CK_{\lambda'}})\circ \Delta^{CK_{\lambda'}}(x)\\
&=(\id \otimes \Delta^{CK_\lambda})\circ \Delta^{CK_\lambda}(x).
\end{align*}
Hence, $\Delta_\lambda$ is coassociative, and $\h_{\D,\T}^{CK_\lambda}$ is a Hopf algebra. \end{proof}

\begin{example}
Let us fix a subset $\T'$ of $\T$ and choose $(\lambda_t)_{t\in \T}$ such that:
\[\lambda_t=\begin{cases}
1\mbox{ if }t\in \T',\\
0\mbox{ otherwise}.
\end{cases}\]
For any tree $T\in \bfT_{\D,\T}$, let us denote by $\adm_{\T'}(T)$ the set of admissible cuts $c$ of $T$ such that
the type of any edge in $c$ belongs to $\T'$. Then:
\begin{align*}
\Delta^{CK_\lambda}(T)&=T\otimes 1+1\otimes T+\sum_{c\in \adm_{\T'}(T)} R^c(T)\otimes P^c(T).
\end{align*}\end{example}

\begin{remark} \begin{enumerate}
\item If $\T=\{t\}$ and $\lambda_t=1$, we recover the usual Connes-Kreimer Hopf algebra of $\D$-decorated rooted trees, 
which we denote by $\h_{\D}^{CK}$, and its duality with the Grossman-Larson Hopf algebra
\cite{ConnesKreimer,Hoffman,Panaite}.
\item If $\T$ and $\D$ are finite, for any $\lambda\in \K^\T$, both $\h_{\D,\T}^{CK_\lambda}$ and $\h_{\D,\T}^{GL_\lambda}$ 
are graded Hopf algebra (by the number of vertices),
and their homogeneous components are finite-dimensional. 
Via the pairing $\langle-,-\rangle$, each one is the graded dual of the other.
\end{enumerate} \end{remark}

\subsection{Hochschild cohomology of coalgebras}

For the sake of simplicity, we assume that the set of types $\T$ is finite
and we put $\T=\{t_1,\ldots,t_N\}$.\\

Let $(C,\Delta)$ be a coalgebra and let $(M,\delta_L,\delta_R)$ be a $C$-bicomodule.
We define a complex, dual to the Hochschild complex for algebras, in the following way:
\begin{enumerate}
\item For any $n\geq 0$, $H_n=\mathcal{L}(M,C^{\otimes n})$.
\item For any $L\in H_n$:
\[b_n(L)=(\id\otimes L)\circ \delta_L+\sum_{i=1}^n (-1)^i (\id^{\otimes (i-1)}\otimes \Delta \otimes \id^{\otimes(n-i)})
\circ L+(-1)^{n+1}(L\otimes \id)\circ \delta_R.\]
\end{enumerate}
In particular, one-cocycles are maps $L:M\longrightarrow C$ such that:
\[\Delta \circ L=(\id\otimes L)\circ \delta_L+(L\otimes \id)\circ \delta_R.\]
We shall consider in particular the bicomodule $(M,\delta_L,\delta_R)$ such that:
\begin{align*}
&\forall x\in C,& \begin{cases}
\delta_L(x)=1\otimes x,\\
\delta_R(x)=\Delta(x).
\end{cases} \end{align*}
If $C$ is a bialgebra, then $M^{\otimes N}$ is also a bicomodule:
\begin{align*}
&\forall x_t\in C,&\begin{cases}
\displaystyle \delta_L\left(\bigotimes_{1\leq i \leq N} x_i\right)=1\otimes \bigotimes_{1\leq i \leq N} x_i,\\[5mm]
\displaystyle \delta_R\left(\bigotimes_{1\leq i \leq N} x_i\right)
=\bigotimes_{1\leq i \leq N} x_i^{(1)}\otimes \prod_{1\leq i \leq N} x_i^{(2)}.
\end{cases} \end{align*}
We denote by $\underline{1}=(1)_{t\in \T} \in \K^\T$, and we take $C=\h_{\D,\T}^{CK_{\underline{1}}}$.
One can identify $S\left(Vect(\bfT_{\D,\T})^{\oplus \T}\right)$ and $C^{\otimes N}$,
$x\delta_{T_i}$ being identified with $1^{\otimes (i-1)}\otimes x\otimes 1^{\otimes (n-i)}$
for any $x\in \bfT_{\D,\T}$ and $1\leq i \leq N$. Then for any $d$,
$B_d:C^{\otimes N}\longrightarrow C$ is a 1-cocycle. Moreover, there is a universal property,
proved in the same way as for the Connes-Kreimer's one \cite{ConnesKreimer}:

\begin{theo} \label{theo22}
Let $B$ be a commutative bialgebra and, for any $d\in \D$, let $L_d:C^{\otimes N}\longrightarrow C$
be a 1-cocycle:
\begin{align*}
&\forall d\in \D,\:\forall x_t\in B,&
\Delta \circ L_d\left( \bigotimes_{1\leq i \leq N} x_i\right)
&=1\otimes \bigotimes_{1\leq i \leq N} x_i+ L_d\left( \bigotimes_{1\leq i \leq N} x_i^{(1)}\right)
\otimes \prod_{1\leq i \leq N} x_i^{(2)}.
\end{align*}
There exists a unique bialgebra morphism $\phi:\h_{\D,\T}^{CK_{\underline{1}}}\longrightarrow C$
such that for any $d\in \D$, $\phi\circ L_d=B_d \circ\phi^{\otimes N}$.
\end{theo}

\subsection{Hopf algebra morphisms}

Our aim is, firstly, to construct Hopf algebras morphisms between 
$\h_{\D,\T}^{CK_\lambda}$ and $\h_{\D,\T}^{CK_\mu}$;
secondly, to construct Hopf algebra isomorphisms between $\h_{\D,\T}^{CK_\lambda}$ 
and $\h_{\D'}^{CK}$ for a well-chosen $\D'$.

\begin{prop}
Let $M\in \M_{\T,\T'}(\K)$, $\lambda\in \K^{\T}$. Then $\Phi_M:\h_{\D,\T'}
^{CK_{M^\top \lambda}}\longrightarrow \h_{\D,\T}^{CK_\lambda}$ is a Hopf algebra morphism. 
\end{prop}

\begin{proof} $\Phi_M$ is a obviously an algebra morphism. Let $T\in \bfT_{\D,\T}$.
\begin{align*}
\Delta_\lambda\circ \Phi_M(T)&=\Phi_M(T)\otimes 1+1\otimes \Phi_M(T)\\
&+\sum_{c\in \adm(T)} \prod_{e\in c}\left(\sum_{t\in \T} m_{t,\type(e)}\lambda_t\right) 
\Phi_M(R^c(T))\otimes \Phi_M(P^c(T))\\
&=\Phi_M(T)\otimes 1+1\otimes \Phi_M(T)\\
&+\sum_{c\in \adm(T)} \prod_{e\in c}(M^\top \lambda)_{\type(e)} \Phi_M(R^c(T))\otimes \Phi_M(P^c(T))\\
&=(\Phi_M\otimes \Phi_M)\circ \Delta_{M^\top a}(T).
\end{align*}
So $\Phi_M$ is a coalgebra morphism from $\h_{\D,\T'}^{CK_{M^\top \lambda}}$ to $\h_{\D,\T}^{CK_\lambda}$. \end{proof}

\begin{cor} \label{cor24}
Let $\lambda,\mu \in \K^\T$, both nonzero. Then $\h_{\D,\T}^{CK_\lambda}$ and $\h_{\D,\T}^{CK_\mu}$
 are isomorphic Hopf algebras.
\end{cor}

\begin{proof} There exists $M\in \M_\T(\K)$, invertible, such that $M^\top \lambda=\mu$.
Then $\Phi_M$ is an isomorphism between $\h_{\D,\T}^{CK_\mu}$ and $\h_{\D,\T}^{CK_\lambda}$. \end{proof}

\begin{defi}
Let us fix $t_0\in \T$. For any $F \in \bfF_{\D,\T}$, we shall say that $\{T_1,\ldots,T_k\}\triangleleft_{t_0} F$ 
if the following conditions hold:
\begin{itemize}
\item $\{T_1,\ldots,T_k\}$ is a partition of $V(F)$. Consequently, for any $i\in [k]$, $T_i\in \bfF_{\D,\T}$, 
by restriction.
\item For any $i\in [k]$, $T_i\in \bfT_{\D,\T}^{(t_0)}$.
\end{itemize}
If $\{T_1,\ldots,T_k\}\triangleleft_{t_0} F$, we denote by $F/\{T_1,\ldots,T_k\}$ 
the forest obtained by contracting $T_i$ to a single vertex
for any $i\in [k]$, decorating this vertex by $T_i$, and forgetting the type of the remaining edges. 
Then $F/\{T_1,\ldots,T_k\}$
is a $\T_{\D,\T}^{(t_0)}$-decorated forest.
\end{defi}

\begin{prop} \label{prop26}
Let $\lambda\in \K^\T$, $t_0\in \T$. Let us consider the map:
\[\Psi_{t_0}^*:\left\{\begin{array}{rcl}
\h_{\D,\T}&\longrightarrow&\h_{\bfT_{\D,\T}^{(t_0)}}\\
F\in \bfF_{\D,\T}&\longrightarrow&\displaystyle \sum_{\{T_1,\ldots,T_k\}\triangleleft_{t_0} F} 
\left(\prod_{e\in E(F)\setminus \sqcup E(T_i)}\lambda_{\type(e)}\right) F/\{T_1,\ldots,T_k\}.
\end{array}\right.\]
Then $\Psi_{t_0}^*$ is a Hopf algebra morphism from $\h_{\D,\T}^{CK_\lambda}$ to $\h_{\bfT_{\D,\T}^{(t_0)}}^{CK}$. 
It is an isomorphism if, and only if, $\lambda_{t_0}\neq 0$.
\end{prop}

\begin{proof}
\textit{First case.} We first assume that $\D$ and $\T$ are finite. In this case, $\h_{\D,\T}^{CK_\lambda}$ 
is the graded dual of $\h_{\D,\T}^{GL_\lambda}$,
with the Hopf pairing $\langle-,-\rangle$; grading $\h_{\bfT_{\D,\T}^{(t_0)}}$ by the number of vertices of the decorations,
$\h_{\bfT_{\D,\T}^{(t_0)}}^{CK}$ is the graded dual of $\h_{\bfT_{\D,\T}^{(t_0)}}^{GL}$.
Moreover, $\Psi_{t_0}^*$ is the transpose of $\Psi_{t_0}$ of proposition \ref{prop19}, so is a Hopf algebra morphism.
If $\lambda_{t_0}\neq 0$, $\Psi_{t_0}$ is an isomorphism, so $\Psi_{t_0}^*$ also is.\\

\textit{General case}. Let $x,y\in \h_{\D,\T}$. There exist finite $\D',\T'$, such that $x,y\in \h_{\D',\T'}$; we can assume that $t_0\in \T'$.
We denote by $\lambda'=\lambda_{\mid \T'}$. 
Then, by the preceding case, denoting by $\Psi'_{t_0}$ the restriction of $\Psi^*_{t_0}$
to $\h_{\D',\T'}$:
\begin{align*}
\Psi^*_{t_0}(xy)&=\Psi'_{t_0}(xy)=\Psi'_{t_0}(x)\Psi'_{t_0}(y)=\Psi^*_{t_0}(x)\Psi^*_{t_0}(y),\\
\Delta^{CK_\lambda}\circ \Psi^*_{t_0}(x)&=\Delta^{CK_{\lambda'}} 
\circ \Psi'_{t_0}(x)=(\Psi'_{t_0}\otimes \Psi'_{t_0})\circ \Delta^{CK_{\lambda'}}(x)
=(\Psi^*_{t_0}\otimes \Psi^*_{t_0})\circ \Delta^{CK_\lambda}(x),
\end{align*}
so $\Psi$ is a Hopf algebra morphism.\\

Let us assume that $\lambda_{t_0}\neq 0$. If $\Psi^*_{t_0}(x)=0$, then $\Psi'_{t_0}(x)=0$. As $a'_{t_0}\neq 0$, by the first case,
$x=0$, so $\Psi^*_{t_0}$ is injective. Moreover, there exists $z\in \h_{\D',\T'}$, such that $\Psi'_{t_0}(z)=y$;
so $\Psi^*_{t_0}(z)=y$, and $\Psi^*_{t_0}$ is surjective.\\

Let us assume that $\lambda_{t_0}=0$. Let $T$ be a tree with two vertices, such that its unique edge is of type $t_0$.
As $T \notin \bfT_{\D,\T}^{(t_0)}$, $\Phi_{t_0}(T)$ has a unique term, given by the partition $X=\{\{x_1\},\{x_2\}\}$, 
where $x_1$ and $x_2$ are the vertices of $T$. 
Hence:
\[\Psi^*_{t_0}(T)=\lambda_{t_0}T'=0,\]
so $\Psi^*_{t_0}$ is not injective.  \end{proof}

\begin{example}
Here, $\T$ contains two elements, $\typ{10}$	 and $\typ{2}$.
In order to simplify, we omit the decorations of vertices. We put:
\begin{align*}
x&=\tun,&y&=\tdeux{2},&z&=\ttroisun{2}{2},&
u&=\tdtroisdeux{}{}{}{2}{10},&v&=\tdtroisdeux{}{}{}{2}{2}.
\end{align*}
Applying $\Psi^*_{\typ{10}}$:
\begin{align*}
\Psi^*_{\typ{10}}(\tun)&=\tdun{$x$},&
\Psi^*_{\typ{10}}(\ttroisun{10}{10})&=
\lambda_{\typ{10}}^2 \tdtroisun{$x$}{$x$}{$x$}{10}{10},\\
\Psi^*_{\typ{10}}(\tdeux{10})&=\lambda_{\typ{10}} \tddeux{$x$}{$x$}{10},&
\Psi^*_{\typ{10}}(\tdtroisdeux{}{}{}{10}{10})&=
\lambda_{\typ{10}}^2\tdtroisdeux{$x$}{$x$}{$x$}{10}{10},\\
\Psi^*_{\typ{10}}(\tdeux{2})&=\lambda_{\typ{2}} \tddeux{$x$}{$x$}{10}+\tdun{$y$},&
\Psi^*_{\typ{10}}(\tdtroisdeux{}{}{}{10}{2})&=\lambda_{\typ{10}}\lambda_{\typ{2}}
\tdtroisdeux{$x$}{$x$}{$x$}{10}{10}+\lambda_{\typ{10}}\tddeux{$x$}{$y$}{10},\\
\Psi^*_{\typ{10}}(\ttroisun{2}{10})&=
\lambda_{\typ{10}}\lambda_{\typ{2}} \tdtroisun{$x$}{$x$}{$x$}{10}{10}
+\lambda_{\typ{10}}\tddeux{$y$}{$x$}{10},&
\Psi^*_{\typ{10}}(\tdtroisdeux{}{}{}{2}{10})&=\lambda_{\typ{10}}\lambda_{\typ{2}}
\tdtroisdeux{$x$}{$x$}{$x$}{10}{10}+\lambda_{\typ{10}}\tddeux{$y$}{$x$}{10}+\tdun{$u$},\\
\Psi^*_{\typ{10}}(\ttroisun{2}{2})&=
\lambda_{\typ{2}}^2 \tdtroisun{$x$}{$x$}{$x$}{10}{10}
+2\lambda_{\typ{2}}\tddeux{$y$}{$x$}{10}+\tdun{$z$},&
\Psi^*_{\typ{10}}(\tdtroisdeux{}{}{}{2}{2})&=\lambda_{\typ{2}}^2
\tdtroisdeux{$x$}{$x$}{$x$}{10}{10}+\lambda_{\typ{2}}\tddeux{$y$}{$x$}{10}
+\lambda_{\typ{2}}\tddeux{$x$}{$y$}{10}+\tdun{$v$}.
\end{align*}\end{example}

\begin{remark}
Although it is not indicated, $\Psi_{t_0}$ and $\Psi_{t_0}^*$ depend on $\lambda$.
\end{remark}

\subsection{Bialgebras in cointeraction}

By \cite{FoissyOperads}, for any $\lambda\in \K^{(\T)}$, 
the operad morphism $\theta_a:\mathbf{PreLie}\longrightarrow \P_\T$,  which send $\bullet$ to $\bullet_\lambda$,
where $\bf{PreLie}$ is the operad of pre-Lie algebras, 
induces a pair of cointeracting bialgebras for any finite set $\D$. 
By construction, the first bialgebra of the pair is $\h_{\D,\T}^{CK_\lambda}$. Let us describe the second one.

\begin{defi}
Let $F\in \bfF_{\T,\D}$. We shall say that $\{T_1,\ldots,T_k\}\triangleleft F$ if:
\begin{enumerate}
\item $\{T_1,\ldots,T_k\}$ is a partition of $V(F)$. 
Consequently, for any $i\in [k]$, $T_i\in \bfF_{\D,\T}$, by restriction.
\item For any $i\in [k]$, $T_i\in \bfT_{\D,\T}$.
\end{enumerate}
If $\{T_1,\ldots,T_k\}\triangleleft F$ and $\dec:[k]\longrightarrow \D$, we denote by $(F/\{T_1,\ldots,T_k\},\dec)$ 
the forest obtained by contracting $T_i$ to a single vertex, 
and decorating this vertex by $\dec(i)$, for all $i\in [k]$. This is an element of $\bfF_{\D,\T}$. 
\end{defi}

\begin{prop}
If $\D$ is finite, $\h_{\D,\T}'$ is the free commutative algebra generated by pairs $(T,d)$, 
where $T\in \bfT_{\T,\D}$ and $d\in \D$. 
The coproduct is given, for any $F\in \bfF_{\D,\T}$, $d\in \D$, by:
\begin{align*}
\delta(F,d)=\sum_{\{T_1,\ldots,T_k\}\triangleleft F}
\sum_{\dec:[k]\longrightarrow \D} ((F/\{T_1,\ldots,T_k\},\dec),d)\otimes (T_1,\dec(1))\ldots (T_k,\dec(k)).
\end{align*}
Then $(\h_{\D,\T}',m,\delta)$ is a bialgebra, and $\h_{\D,\T}^{CK_\lambda}$ is a coalgebra 
in the category of $\h_{\D,\T}'$-comodules via the coaction given, for any $T\in \bfT_{\D,\T}$, by:
\begin{align*}
\overline{\delta}(T)=\sum_{\{T_1,\ldots,T_k\}\triangleleft T}
\sum_{\dec:[k]\longrightarrow \D} ((T/\{T_1,\ldots,T_k\},\dec)\otimes (T_1,\dec(1))\ldots (T_k,\dec(k)).
\end{align*}\end{prop}

\begin{cor}
Let us assume that $\D$ is given a semigroup law denoted by $+$. 
If $F\in \bfF_{\T,\D}$, and $\{T_1,\ldots,T_k\}\triangleleft F$, then naturally 
$T_i \in \bfT_{\T,\D}$ for any $i$ and the $\T$-typed forest $F/\{T_1,\ldots,T_k\}$
is given a $\D$-decoration, decorating the vertex obtained in the contradiction of $T_i$
by the sum of the decorations of the vertices of $T_i$. Then $\h_{\D,\T}$ is given a second coproduct $\delta$
such that for any $F\in \bfF_{\D,\T}$:
\begin{align*}
\delta(F)&=\sum_{\{T_1,\ldots,T_k\}\triangleleft F} F/\{T_1,\ldots,T_k\}\otimes T_1\ldots T_k.
\end{align*}
Then $(\h_{\D,\T},m,\delta)$ is a bialgebra and $\h_{\D,\T}^{CK_\lambda}$ is a coalgebra 
in the category of $\h_{\D,\T}$-comodules via the coaction $\delta$.
\end{cor}

\begin{proof}
We denote by $I$ the ideal of $\h'_{\D,\T}$ generated by pairs $(T,d)$ such that $T\in \bfT_{\D,\T}$
and $d\in \D$, with:
\[d\neq \sum_{v\in V(T)}\dec(v).\]
The quotient $\h'_{\D,\T}/I$ is identified with $\h_{\D,\T}$, trough the surjective algebra morphism:
\[\varpi:\left\{\begin{array}{rcl}
\h'_{\D,\T}&\longrightarrow&\h_{\D,\T}\\
(F,d)\in \bfF_{\D,\T}\times \D&\longrightarrow&\begin{cases}
\displaystyle F\mbox{ if }d=\sum_{v\in V(F)}\dec(v),\\
0\mbox{ otherwise}.
\end{cases}
\end{array}\right.\]
Let us prove that $I$ is a coideal. Let $T\in \bfT_{\T,\D}$, $d\in \D$,
$\{T_1,\ldots,T_k\}\triangleleft F$, $\dec:[k]\longrightarrow \D$ such that 
$((T/\{\T_1,\ldots,T_k\},\dec),d)\notin I$ and for any $i$, $(T_i,\dec(i))\notin I$. Then:
\begin{align*}
&\forall i\in [k],& \sum_{v\in V(T_i)} \dec(v)&=\dec(i),&
&&\sum_{i=1}^k \dec(i)&=d.
\end{align*} 
Hence:
\[\sum_{v\in V(T)}\dec(v)=\sum_{i=1}^k \sum_{v\in V(T_i)} \dec(v)
=\sum_{i=1}^k \dec(i)=d,\]
so $(T,d)\notin I$. Consequently, if $T\in I$, then 
$((T/\{\T_1,\ldots,T_k\},\dec),d)\in I$ or at least one of the $(T_i,\dec(i))$ belongs to $I$.
Hence:
\[\delta(I)\subseteq I\otimes \h'_{\D,\T}+\otimes \h'_{\D,\T}\otimes I.\]
So $I$ is a coideal. The coproduct induced on $\h_{\D,\T}$ by the morphism $\varpi$
is precisely the one given in the setting of this Corollary. \end{proof}

In particular, if $\D$ is reduced to a single element, denoted by $*$,
if we give it its unique semigroup structure ($*+*=*$), We obtain again the result of \cite{Manchon}.

\subsection{The Bruned-Hairer-Zambotti construction}

We now consider the coproducts on typed trees in \cite[Theorem 2.2.16]{Bruned},
the first one with $\overline{\mathfrak{A}}(F)=\mathfrak{A}(F)$ and the second one with
$\overline{\mathfrak{A}}(F)=\mathfrak{A}^+(F)$ of \cite[Definition 2.4.1]{Bruned}. 
By definition \cite[Definition 2.26]{Bruned} of admissible subtrees, according to the notations we choose in this paper:
\begin{itemize}
\item Let $\L$ be a finite set of types. The considered trees are $\L$-typed and the leaves are $\L$-decorated.
Considering that the internal vertices of such a tree are in fact decorated by an element $0\notin \L$,
these trees form a subset of $\bfT_{\L\sqcup\{0\},\L}$ which we denote by $\bfT_\L'$. 
\item The first coproduct $\Delta_+$ is given on any tree $T\in \bfT_\L'$ by:
\[\Delta_+(T)=\sum_{c\in \adm'(T)} R^c(T)\otimes P^c(T),\]
where $\adm'(T)$ is the set of admissible cuts of $T$ such that the set of leaves of $R^c(T)$ is a subset of leaves of $T$
(note that automatically, the leaves of $P^c(T)$ are also leaves of $T$).
\item The second coproduct is given on any tree $T\in \bfT_\L'$ by:
\[\overline{\Delta}(T)=\sum_{\{T_1,\ldots,T_k\}\triangleleft' T}T/\{T_1,\ldots,T_k\}\otimes T_1\ldots T_k,\]
where the sum runs over all $\{T_1,\ldots,T_k\}\triangleleft T$ such that the leaves of $T/\{T_1,\ldots,T_k\}$
are leaves of $T$. 
\end{itemize}
We denote by $\h'_\L$ the subalgebra of $\h_{\L\sqcup\{0\}, \L}$ generated by $bfT_\L'$.
Then $(\h'_\L,m,\Delta_+)$ and $(\h'_\L,m,\delta)$ are bialgebras.

Let us give to $\L$ any product $\times$ making it a commutative associative semigroup.
We extend this structure to $\L\sqcup \{0\}$ by:
\begin{align*}
&\forall t\in \L\sqcup \{0\},& 0\times t=t\times 0=0.
\end{align*}
We take $\lambda_t=1$ for any $t\in \L\sqcup\{0\}$, and obtain with this data two coproducts
$\Delta$ and $\delta$ on $\h_{\L\sqcup\{0\}, \L}$, the first one given by admissible cuts and the second one by contractions
of subtrees. The subalgebra $\h''_\L$ of  $\h_{\L\sqcup\{0\}, \L}$ generated by trees such that any internal vertex
is decorated by $0$ (note that the leaves of such a tree are decorated by $L\sqcup \{0\}$)  is a subbialgebra for both coproducts.
We denote by $I$ the ideal of $\h''_\L$ generated by trees with at least one leaf decorated by $0$. 
Then it is a coideal for both coproducts, so the quotient algebra $\h''_\L/I$ inherits two coproducts, still denoted by $\Delta$
and $\delta$.  This algebra is trivially identified with the algebra generated by $\bfT_\L'$, that is to say
with $\h'_\L$. By construction of the different coproducts, this identification is an isomorphism from the Hopf algebra
$(\h''_\L/I,\Delta)$ to $(\h'_\L/I,\Delta_+)$  and from the Hopf algebra $(\h''_\L/I,\delta)$ to $(\h'_\L/I,\overline{\Delta})$. 
In other words, the Bruned-Hairer-Zambotti construction of cointeracting bialgebras on typed trees is a subquotient
of the construction presented here.

\section{Appendix}

\subsection{Values of $t_{D,T}(k)$}

For $D=1$:
\[\begin{array}{|c||c|c|c|c|c|c|c|c|}
\hline T\setminus k&1&2&3&4&5&6&7&8\\
\hline \hline 1& 1& 1& 2& 4& 9& 20& 48& 115\\
\hline 2&1& 2& 7& 26& 107& 458& 2058& 9498\\
\hline 3&1& 3& 15& 82& 495& 3144& 20875& 142773\\
\hline 4&1& 4& 26& 188& 1499& 12628& 111064& 1006840\\
\hline 5&1& 5& 40& 360& 3570& 37476& 410490& 4635330\\
\hline 6&1& 6& 57& 614& 7284& 91566& 1200705& 16232820\\
\hline 7&1& 7& 77& 966& 13342& 195384& 2984142& 46990952\\
\hline 8&1& 8& 100& 1432& 22570& 377320& 6578116& 118238600\\
\hline 9&1& 9& 126& 2028& 35919& 674964& 13225632& 267188229\\
\hline 10&1& 10& 155& 2770& 54465& 1136402& 24723000& 554540590\\
\hline\end{array}\]
For $D=2$:
\[\begin{array}{|c||c|c|c|c|c|c|c|c|}
\hline T\setminus k&1&2&3&4&5&6&7&8\\
\hline 1&2& 4& 14& 52& 214& 916& 4116& 18996\\
\hline 2&2& 8& 52& 376& 2998& 25256& 222128& 2013680\\
\hline 3&2& 12& 114& 1228& 14568& 183132& 2401410& 32465640\\
\hline 4&2& 16& 200& 2864& 45140& 754640& 13156232& 236477200\\
\hline 5&2& 20& 310& 5540& 108930& 2272804& 49446000& 1109081180\\
\hline 6&2& 24& 444& 9512& 224154& 5606520& 146204792& 3930863232\\
\hline 7&2& 28& 602& 15036& 413028& 12043500& 366122190& 11475005616\\
\hline 8&2& 32& 784& 22368& 701768& 23373216& 811575408& 29052861280\\
\hline 9&2& 36& 990& 31764& 1120590& 41969844& 1638712716& 65965167108\\
\hline 10&2& 40& 1220& 43480& 1703710& 70875208& 3073688160& 137426005200\\
\hline\end{array}\]
For $D=3$:
\[\begin{array}{|c||c|c|c|c|c|c|c|c|}
\hline T\setminus k&1&2&3&4&5&6&7&8\\
\hline 1&3& 9& 45& 246& 1485& 9432& 62625& 428319\\
\hline 2&3& 18& 171& 1842& 21852& 274698& 3602115& 48698460\\
\hline 3&3& 27& 378& 6084& 107757& 2024892& 39676896& 801564687\\
\hline 4&3& 36& 666& 14268& 336231& 8409780& 219307188& 5896294848\\
\hline 5&3& 45& 1035& 27690& 814680& 25444584& 828506340& 27812997990\\
\hline 6&3& 54& 1485& 47646& 1680885& 62954766& 2458069074& 98947750662\\
\hline 7&3& 63& 2016& 75432& 3103002& 135520812& 6170116638& 289616448690\\
\hline 8&3& 72& 2628& 112344& 5279562& 263423016& 13701398868& 734709311208\\
\hline 9&3& 81& 3321& 159678& 8439471& 473586264& 27703353159& 1670715963729\\
\hline 10&3& 90& 4095& 218730& 12842010& 800524818& 52018920345& 3484841027040\\
\hline\end{array}\]
For $D=4$:
\[\begin{array}{|c||c|c|c|c|c|c|c|c|}
\hline T\setminus k&1&2&3&4&5&6&7&8\\
\hline 1&4& 16& 104& 752& 5996& 50512& 444256& 4027360\\
\hline 2&4& 32& 400& 5728& 90280& 1509280& 26312464& 472954400\\
\hline 3&4& 48& 888& 19024& 448308& 11213040& 292409584& 7861726464\\
\hline 4&4& 64& 1568& 44736& 1403536& 46746432& 1623150816& 58105722560\\
\hline 5&4& 80& 2440& 86960& 3407420& 141750416& 6147376320& 274852010400\\
\hline 6&4& 96& 3504& 149792& 7039416& 351230688& 18268531824& 979612414944\\
\hline 7&4& 112& 4760& 237328& 13006980& 756866096& 45910215120& 2871018269632\\
\hline 8&4& 128& 6208& 353664& 22145568& 1472317056& 102037088448& 7290356719488\\
\hline 9&4& 144& 7848& 502896& 35418636& 2648533968& 206451156768& 16590568445280\\
\hline 10&4& 160& 9680& 689120& 53917640& 4479065632& 387863411920& 34625886677920\\
\hline\end{array}\]

\subsection{Values of $t'_{D,T}(k)$}

For $D=1$:
\[\begin{array}{|c||c|c|c|c|c|c|c|c|}
\hline T\setminus k&1&2&3&4&5&6&7&8\\
\hline 1&1& 0& 0& 0& 0& 0& 0& 0\\ 
\hline 2&1& 1& 3& 10& 39& 160& 702& 3177\\ 
\hline 3&1& 2& 9& 46& 268& 1660& 10845& 73270\\ 
\hline 4&1& 3& 18& 124& 963& 7968& 69236& 621999\\ 
\hline 5&1& 4& 30& 260& 2525& 26136& 283528& 3178696\\ 
\hline 6&1& 5& 45& 470& 5480& 68096& 885805& 11904160\\ 
\hline 7&1& 6& 63& 770& 10479& 151956& 2304974& 36110880\\ 
\hline 8&1& 7& 84& 1176& 18298& 303296& 5255964& 94051770\\ 
\hline 9&1& 8& 108& 1704& 29838& 556464& 10845732& 218239560\\ 
\hline 10&1& 9& 135& 2370& 46125& 955872& 20696076& 462558987\\
\hline\end{array}\]
For $D=2$:
\[\begin{array}{|c||c|c|c|c|c|c|c|c|}
\hline T\setminus k&1&2&3&4&5&6&7&8\\
\hline 1&2& 0& 0& 0& 0& 0& 0& 0\\ 
\hline 2&2& 4& 22& 144& 1090& 8864& 76162& 678532\\ 
\hline 3&2& 8& 68& 688& 7886& 96896& 1250780& 16713504\\ 
\hline 4&2& 12& 138& 1888& 29004& 476736& 8213588& 146342376\\ 
\hline 5&2& 16& 232& 4000& 77060& 1586304& 34185344& 761389360\\ 
\hline 6&2& 20& 350& 7280& 168670& 4171744& 107932710& 2884827980\\ 
\hline 7&2& 24& 492& 11984& 324450& 9370368& 282934428& 8822987856\\
\hline 8&2& 28& 658& 18368& 569016& 18793600& 648698792& 23119514576\\ 
\hline 9&2& 32& 848& 26688& 930984& 34609920& 1344232416& 53898191520\\ 
\hline 10&2& 36& 1062& 37200& 1442970& 59627808& 2573660298& 114661732500\\
\hline\end{array}\]
For $D=3$:
\[\begin{array}{|c||c|c|c|c|c|c|c|c|}
\hline T\setminus k&1&2&3&4&5&6&7&8\\
\hline 1&3& 0& 0& 0& 0& 0& 0& 0\\ 
\hline 2&3& 9& 72& 705& 7947& 96588& 1237728& 16450389\\
\hline 3&3& 18& 225& 3408& 58347& 1072224& 20685195& 413084610\\ 
\hline 4&3& 27& 459& 9405& 216081& 5315112& 136987407& 3650993163\\ 
\hline 5&3& 36& 774& 19992& 576405& 17763984& 572991726& 19100718828\\ 
\hline 6&3& 45& 1170& 36465& 1264950& 46852884& 1815034140& 72635168880\\ 
\hline 7&3& 54& 1647& 60120& 2437722& 105455952& 4768982442& 222723271080\\ 
\hline 8&3& 63& 2205& 92253& 4281102& 211832208& 10953036318& 584744300226\\ 
\hline 9&3& 72& 2844& 134160& 7011846& 390570336& 22727284344& 1365242802048\\ 
\hline 10&3& 81& 3564& 187137& 10877085& 673533468& 43560017892& 2907844041231\\
\hline\end{array}\]
For $D=4$:
\[\begin{array}{|c||c|c|c|c|c|c|c|c|}
\hline T\setminus k&1&2&3&4&5&6&7&8\\
\hline 1&4& 0& 0& 0& 0& 0& 0& 0\\ 
\hline 2&4& 16& 168& 2192& 32844& 531200& 9051376& 159962784\\ 
\hline 3&4& 32& 528& 10656& 242792& 5939968& 152518064& 4053650976\\ 
\hline 4&4& 48& 1080& 29488& 902100& 29551104& 1014147872& 35989518528\\ 
\hline 5&4& 64& 1824& 62784& 2411024& 98976256& 4252211232& 188790415552\\ 
\hline 6&4& 80& 2760& 114640& 5297820& 261422336& 13491005840& 719200139360\\ 
\hline 7&4& 96& 3888& 189152& 10218744& 588999936& 35487727184& 2208096700896\\ 
\hline 8&4& 112& 5208& 290416& 17958052& 1184031744& 81574704960& 5802692175744\\ 
\hline 9&4& 128& 6720& 422528& 29428000& 2184360960& 169377005376& 13557899008896\\ 
\hline 10&4& 144& 8424& 589584& 45668844& 3768659712& 324805399344& 28894042642464\\
\hline\end{array}\]

\bibliographystyle{amsplain}
\bibliography{biblio}

\end{document}